\documentclass[12pt]{amsart}

\usepackage{amsfonts, amsthm, amsmath}

\usepackage{rotating}

\usepackage{tikz}

\usepackage{graphics}

\usepackage{amssymb}

\usepackage{amscd}

\usepackage[latin2]{inputenc}

\usepackage{t1enc}

\usepackage[mathscr]{eucal}

\usepackage{indentfirst}

\usepackage{graphicx}

\usepackage{graphics}

\usepackage{pict2e}

\usepackage{mathrsfs}

\usepackage{enumerate}
\usepackage[pagebackref]{hyperref}
\hypersetup{colorlinks=true}
\usepackage{cite}
\usepackage{color}
\usepackage{epic}
\usepackage{hyperref} 
\usepackage{framed}
\usepackage{mathabx}
\usepackage{booktabs}
\usepackage{makecell}
\newcolumntype{V}{!{\vrule width 2pt}}

\numberwithin{equation}{section}
\def\blue{\textcolor{blue}}
\def\red{\textcolor{red}}
\def\magenta{\textcolor{magenta}}

\def\cyan{\textcolor{cyan}}

\theoremstyle{plain}

\newtheorem{theorem}{Theorem}[section]
\newtheorem{corollary}[theorem]{Corollary}
\newtheorem{proposition}[theorem]{Proposition}

\newtheorem{remark}[theorem]{Remark}
\newtheorem{lemma}[theorem]{Lemma}
\newtheorem{definition}[theorem]{Definition}

\newtheorem{problem}[theorem]{Problem}

\newcommand{\leaf}{\mathsf{leaf}}
\newcommand{\intt}{\mathsf{int}}
\def\deg{\mathsf{deg}}

\def\run{\mathsf{run}}
\def\hrun{\mathsf{hrun}}
\def\vrun{\mathsf{vrun}}
\def\ret{\mathsf{ret}}

\def\bran{\mathsf{bran}}
\def\larm{\mathsf{larm}}
\def\rarm{\mathsf{rarm}}

\def\peak{\mathsf{peak}}

\def\spi{\mathsf{spi}}
\def\lrb{\mathsf{lrb}}

\def\rspi{\mathsf{rspi}}

\def\iest{\mathsf{iest}}

\def\Par{\mathsf{Par}}
\def\Bro{\mathsf{Bro}}
\def\Yle{\mathsf{Yle}}
\def\Lev{\mathsf{Lev}}
\def\Lsw{\mathsf{Lsw}}

\def\rev{\mathsf{rev}}

\def\lsw{\mathsf{lsw}}
\def\rsw{\mathsf{rsw}}
\def\sat{\mathsf{st}}
\def\lc{\mathsf{lc}}
\def\Lc{\mathsf{Lc}}
\def\dlev{\mathsf{dlev}}
\def\lchain{\mathsf{lchain}}
\def\dlsw{\mathsf{dlsw}}
\def\drsw{\mathsf{drsw}}

\def\SS{\mathfrak{S}}
\def\D{\mathcal{D}}

\def\B{\mathcal{B}}

\def\JR{\mathcal{J}}

\def\P{\mathcal{P}}

\def\omi{\mathsf{omi}}
\def\top{\mathsf{top}}
\def\rpop{\mathsf{rpop}}
\def\lev{\mathsf{lev}}
\def\lsw{\mathsf{lsw}}
\def\IR{\mathsf{IR}}
\def\IL{\mathsf{IL}}

\def\is{\mathsf{is}}
\def\ds{\mathsf{ds}}

\def\N{ \mathbb{N}}

\newcommand{\overbar}[1]{\,\,\mkern -1.5mu\overline{\mkern -4mu#1\mkern-1.5mu}\mkern 1.5mu}

\begin{document}

\title[Two involutions on binary trees]{Two involutions on binary trees and generalizations}


\author[Y. Li]{Yang Li}
\address[Yang Li]{Research Center for Mathematics and Interdisciplinary Sciences, Shandong University \& Frontiers Science Center for Nonlinear Expectations, Ministry of Education, Qingdao 266237, P.R. China}
\email{202117113@mail.sdu.edu.cn}

\author[Z. Lin]{Zhicong Lin}
\address[Zhicong Lin]{Research Center for Mathematics and Interdisciplinary Sciences, Shandong University \& Frontiers Science Center for Nonlinear Expectations, Ministry of Education, Qingdao 266237, P.R. China}
\email{linz@sdu.edu.cn}


\author[T. Zhao]{Tongyuan Zhao}
\address[Tongyuan Zhao]{College of Science, China University of Petroleum, 102249 Beijing, P.R. China}
\email{zhaotongyuan@cup.edu.cn}

\date{\today}

\begin{abstract}
This paper investigates two involutions on binary trees. One is the mirror symmetry of binary trees which combined  with the classical bijection $\varphi$ between  binary trees and plane trees answers an open problem posed by Bai and Chen. This involution can be generalized   to weakly increasing trees, which admits to merge two recent equidistributions found by Bai--Chen and Chen--Fu, respectively. The other one is constructed to answer a bijective problem on di-sk trees asked by Fu--Lin--Wang and can be generalized naturally to rooted labeled trees. This second involution combined with $\varphi$  leads to a new statistic on plane trees whose distribution  gives the Catalan's triangle. Moreover, a quadruple equidistribution on plane trees involving this new statistic is proved via a recursive bijection. 
\end{abstract}


\keywords{Binary trees; Plane trees; Involutions; Bijections; Weakly increasing trees; Rooted labeled trees}

\maketitle


\section{Introduction}\label{sec1: intro}

Plane trees and binary trees are two of the most fundamental objects in the garden of Catalan numbers~\cite[Chapter~1.5]{St1}. The main objective of this paper is to study two involutions on binary trees (or plane trees), their applications and generalizations. One is the mirror symmetry of binary trees and the other one is newly constructed, which was motivated by a bijective problem on di-sk trees.

Recall that a {\em plane tree} is a rooted tree in which the children of each node are linearly ordered. 
A node without any child is called a {\em leaf}, and an {\em internal node} otherwise. 
For a node $v$ of a plane tree $T$, the {\em level of $v$} is the distance between $v$ and the root of $T$, i.e., the number of edges in the unique path from $v$ to the root. Let $\alpha(j_1,j_2,\ldots,j_h)$ be the number of plane trees having $j_i$ nodes in level $i$ for $1\leq i\leq h$. This number has the following neat binomial expression (see Flajolet~\cite{Fla} for a  continued fraction proof):  
$$
\alpha(j_1,j_2,\ldots,j_h)={j_1+j_2-1\choose j_1-1}{j_2+j_3-1\choose j_2-1}\cdots{j_{h-1}+j_h-1\choose j_{h-1}-1}.
$$

In the study of  decompositions of 132-avoiding permutations, Bai and Chen~\cite{BC} introduced a statistic, called the right spanning widths of nodes, on plane trees that is equally distributed as the levels of nodes. 
The {\em right spanning width of a node $v$} in a plane tree $T$, denoted $\rsw_T(v)$, is the number of children of $v$ plus the number of edges attached to other nodes (than $v$) on the path from $v$ to the root  from the right-hand side (of the path); see Fig.~\ref{RSW} for an illustration.
\begin{figure}
\centering
\begin{tikzpicture}[scale=0.3]
\draw[-](0,9) to (-3,6);\draw[very thick,blue](0,9) to (0,6);\draw[very thick,blue](0,9) to (3,6);
\draw[-](-3,6) to (-6,3);\draw[very thick,blue](-3,6) to (0,3);\draw[-](0,3) to (0,0);
\draw[-](-3,6) to (-3,3);\draw[-](-3,0) to (-3,3);
\draw[very thick,blue](-3,0) to (-5,-3);\draw[very thick,blue](-3,0) to (-1,-3);
\node at  (0,9){$\bullet$};\node at  (-3,6){$\bullet$};\node at  (0,6){$\bullet$};\node at  (3,6){$\bullet$};
\node at  (-6,3){$\bullet$};\node at  (0,3){$\bullet$};\node at  (-3,3){$\bullet$};
\node at  (-3,0){$\bullet$};\node at  (-5,-3){$\bullet$};\node at  (-1,-3){$\bullet$};
\node at  (0,0){$\bullet$};\node at  (-3.5,0.5){$v$};
\end{tikzpicture}
\caption{The right spanning width of $v$ is $5$, which counts the blue bold  edges.\label{RSW}}
\end{figure}
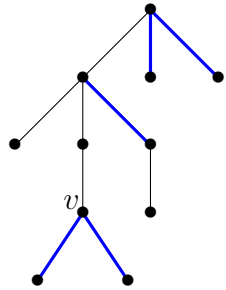
 Their result can be stated as below, which  refines the well-known fact that leaves and internal nodes are equally distributed on plane trees.

\begin{theorem}[Bai and Chen]\label{H:RSW} Fix integers $n,k\geq1$ and a multiset $M$. There are equally many plane trees with $n$ edges and $k$ leaves whose levels constitute $M$ and plane trees with $n$ edges and $k$ internal nodes whose right spanning widths constitute $M$.
\end{theorem}

This equidistribution was proved by combining two bijections, one due to Jani and Rieper~\cite{JR} and another one was newly constructed in~\cite{BC}, between plane trees and 132-avoiding permutations. Recall that a {\em binary tree} is a special type of rooted tree in which every internal node has either one left child  or one right child  or both.  The first purpose of this paper is to show that combining the mirror symmetry on binary trees with the classical bijection between plane trees and binary trees leads to an involution proof of Theorem~\ref{H:RSW}, which can be generalized to the  weakly increasing trees introduced in~\cite{Lin20}. The generalization of this simple involution will also reproves a quintuple equidistribution on weakly increasing trees found by Chen and Fu~\cite{CFu}. 

 The second purpose of this paper is to construct an intriguing involution on binary trees, which was inspired by a bijective problem posed by Fu, Lin and Wang~\cite{FLW} regarding a symmetry on di-sk trees. A {\em di-sk tree} is a binary tree whose nodes are labeled by $\oplus$ or $\ominus$, and no node has the same label as its right child (this requirement is called the {\em right chain condition}). For a binary tree or a di-sk tree $T$, we consider two different orders to traverse the nodes of $T$:
 \begin{itemize}
 \item the {\em preorder}, i.e., recursively traversing the parent to the left subtree then to the right subtree;
 \item the {\em reverse preorder}, i.e., recursively traversing the right subtree  to the left subtree then to the parent. 
 \end{itemize}
 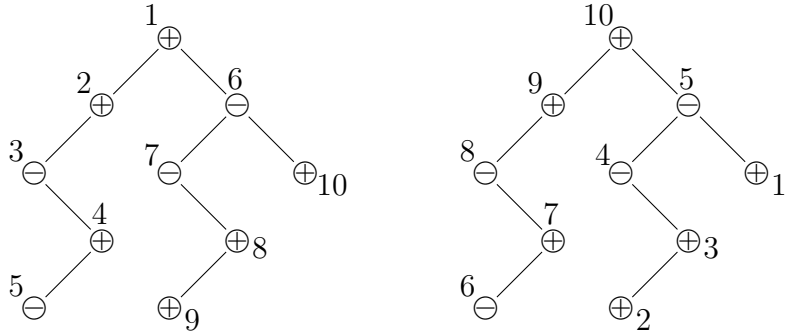
\begin{figure}
\begin{tikzpicture}[scale=0.3]
\draw[-] (17,10) to (19,12);
\draw[-] (20,12) to (22,10);
\draw[-] (22,9) to (20,7);
\draw[-] (23,9) to (25,7);
\draw[-] (20,6) to (22,4);
\draw[-] (22,3) to (20,1);

\node at (19.5,0.5) {$\oplus$};
\node at (20.5,0) {$9$};
\draw[-] (16,9) to (14,7);
\node at (13.5,6.5) {$\ominus$};
\node at (12.7,7.5) {$3$};

\draw[-] (14,6) to (16,4);
\node at (16.5,3.5) {$\oplus$};
\node at (16.4,4.7) {$4$};

\draw[-] (16,3) to (14,1);
\node at (13.5,0.5) {$\ominus$};
\node at (12.7,1.5) {$5$};
\node at (16.5,9.5) {$\oplus$};
\node at (15.7,10.5) {$2$};
\node at (19.5,12.5) {$\oplus$};
\node at (18.7,13.5) {$1$};
\node at (22.5,9.5) {$\ominus$};
\node at (22.4,10.7) {$6$};
\node at (19.5,6.5) {$\ominus$};
\node at (18.7,7.3) {$7$};
\node at (25.5,6.5) {$\oplus$};
\node at (26.7,6) {$10$};
\node at (22.5,3.5) {$\oplus$};
\node at (23.5,3.2) {$8$};

\draw[-] (37,10) to (39,12);
\draw[-] (40,12) to (42,10);
\draw[-] (42,9) to (40,7);
\draw[-] (43,9) to (45,7);
\draw[-] (40,6) to (42,4);
\draw[-] (42,3) to (40,1);

\node at (39.5,0.5) {$\oplus$};
\node at (40.5,0) {$2$};
\draw[-] (36,9) to (34,7);
\node at (33.5,6.5) {$\ominus$};
\node at (32.7,7.5) {$8$};

\draw[-] (34,6) to (36,4);
\node at (36.5,3.5) {$\oplus$};
\node at (36.4,4.7) {$7$};

\draw[-] (36,3) to (34,1);
\node at (33.5,0.5) {$\ominus$};
\node at (32.7,1.5) {$6$};
\node at (36.5,9.5) {$\oplus$};
\node at (35.7,10.5) {$9$};
\node at (39.5,12.5) {$\oplus$};
\node at (38.5,13.5) {$10$};
\node at (42.5,9.5) {$\ominus$};
\node at (42.4,10.7) {$5$};
\node at (39.5,6.5) {$\ominus$};
\node at (38.7,7.3) {$4$};
\node at (45.5,6.5) {$\oplus$};
\node at (46.5,6) {$1$};
\node at (42.5,3.5) {$\oplus$};
\node at (43.5,3.2) {$3$};

\end{tikzpicture}
\caption{The preorder (on left) and the reverse preorder (on right) of a di-sk tree.\label{preorder:disk}}
\end{figure}
See Fig.~\ref{preorder:disk} for the preorder and the reverse preorder of a di-sk tree.  Following~\cite{FLW}, denote by $\top(T)$ (resp.~$\rpop(T)$)  the number of initial $\oplus$-nodes under the preorder (resp.~reverse preorder) in $T$ and let $\omi(T)$ be the number of $\ominus$-nodes in $T$. For example, if $T$ is the di-sk tree in Fig.~\ref{preorder:disk}, then $\top(T)=2, \rpop(T)=3$ and $\omi(T)=4$. Note that di-sk trees are in natural bijection with separable permutations, under which the statistics `$\omi$' and `$\top$' in trees are corresponding respectively to the statistics of descents and components~\cite{FLW} in permutations.   Using generating functions, Fu, Lin and Wang~\cite{FLW} proved the following symmetry 
\begin{equation}\label{top:rpop}
\sum_{T\in\mathfrak{DT}_n}t^{\omi(T)}x^{\top(T)}y^{\rpop(T)}=\sum_{T\in\mathfrak{DT}_n}t^{\omi(T)}x^{\rpop(T)}y^{\top(T)},
\end{equation}
where $\mathfrak{DT}_n$ denotes the set of di-sk trees with $n$ nodes. They also proved several other symmetric distributions on di-sk trees, but the symmetry in~\eqref{top:rpop} is the only one that lacks a combinatorial proof. Thus, they posed the following open problem at the end of the paper. 
\begin{problem}[Fu, Lin and Wang]\label{prob:bij}
Can one find a  bijective proof of the symmetry on di-sk trees in~\eqref{top:rpop}? 
\end{problem}

The rest of this paper is organized as follows. In Section~\ref{sec:2}, we investigate the mirror symmetry of binary trees, which leads to a joint generalization of two equidistributions found recently by Bai--Chen and Chen--Fu, respectively. We also  answer affirmatively an open problem posed by Bai--Chen  and introduce four new tree statistics analogous to $\rsw$, having the same distribution as levels on plane trees. In Section~\ref{sec:3}, we construct a new involution on binary trees, which can be generalized to di-sk trees and rooted labeled trees.  This involution answers Problem~\ref{prob:bij} and leads to a new statistic whose distribution over  plane trees gives the Catalan's triangle. Moreover, a quadruple equidistribution on plane trees involving this new statistic is proved via a recursive bijection.

\section{The mirror symmetry of binary trees and applications}\label{sec:2}
Let $\P_n$ be the set of all plane trees with $n$ edges. 
For the sake of convenience, we will label the nodes of a plane tree in $\P_n$ by using $\{0,1,\ldots,n\}$ in order to distinguish them and we assume that the root is labeled by $0$ and the leftmost child of the root is labeled by $1$. In a plane tree, nodes with the same parent are called {\em siblings} and the siblings to the left (resp.~right) of a node $v$ are called {\em elder} (resp.~{\em younger}) {\em siblings} of $v$.

 Let $\B_n$ be the set of all binary trees with $n$ nodes. For a tree $T\in\P_n$, we define the binary tree $\varphi(T)\in\B_n$ by requiring that for each pair of non-root  nodes $(x,y)$ in $T$: 
 \begin{enumerate}[(i)]
 \item $y$ is the left child of $x$ in $\varphi(T)$ only if  when  $y$ is the leftmost child of  $x$ in $T$;
 \item  $y$ is the right child of $x$ in $\varphi(T)$ only if when $x$ is the closest elder sibling of $y$ in $T$.
 \end{enumerate}
The mapping $T\mapsto\varphi(T)$ establishes a natural one-to-one correspondence between $\P_n$ and $\B_n$, which is well known~\cite[Page~9]{St1}.   See the first step in Fig.~\ref{invo1} for an example of this bijection $\varphi$. 
For a binary tree $B\in\B_n$, let $\phi(B)$ be the  mirror symmetry of $B$. Clearly, the map $B\mapsto\phi(B)$ is an involution on $\B_n$.  Now introduce the involution $\tilde\phi$ on $\P_n$ by $\tilde\phi=\varphi^{-1}\circ\phi\circ\varphi$. See  Fig.~\ref{invo1} for an example of the involution $\tilde\phi$.

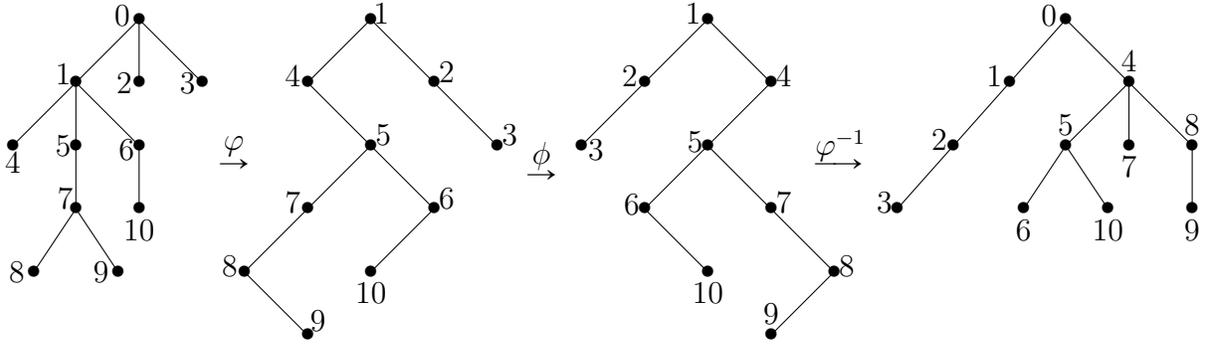
\begin{figure}
\centering
\begin{tikzpicture}[scale=0.28]
\draw[-](0,9) to (-3,6);\draw[-](0,9) to (0,6);\draw[-](0,9) to (3,6);
\draw[-](-3,6) to (-6,3);\draw[-](-3,6) to (0,3);\draw[-](0,3) to (0,0);
\draw[-](-3,6) to (-3,3);\draw[-](-3,0) to (-3,3);
\draw[-](-3,0) to (-5,-3);\draw[-](-3,0) to (-1,-3);
\node at  (0,9){$\bullet$};\node at  (-0.8,9.2){$0$};
\node at  (-3,6){$\bullet$};\node at  (-3.6,6.4){$1$};
\node at  (0,6){$\bullet$};\node at  (-0.7,6){$2$};
\node at  (3,6){$\bullet$};\node at  (2.3,6){$3$};
\node at  (-6,3){$\bullet$};\node at  (-6,2.2){$4$};
\node at  (0,3){$\bullet$};\node at  (-0.6,2.8){$6$};
\node at  (-3,3){$\bullet$};\node at  (-3.6,3){$5$};
\node at  (-3,0){$\bullet$};
\node at  (-5,-3){$\bullet$};\node at  (-5.8,-3){$8$};
\node at  (-1,-3){$\bullet$};\node at  (-1.8,-3){$9$};
\node at  (0,0){$\bullet$};\node at  (0,-1){$10$};
\node at  (-3.5,0.5){$7$};

\node at  (4.5,2){$\rightarrow$};\node at  (4.5,3){$\varphi$};

\draw[-](11,9) to (8,6);
\draw[-](11,9) to (17,3);\draw[-](8,6) to (14,0);
\draw[-](11,-3) to (14,0);\draw[-](11,3) to (5,-3);\draw[-](8,-6) to (5,-3);
\node at  (11,9){$\bullet$};\node at  (11.5,9.3){$1$};
\node at  (8,6){$\bullet$};\node at  (7.3,6.3){$4$};
\node at  (14,6){$\bullet$};\node at  (14.6,6.5){$2$};
\node at  (17,3){$\bullet$};\node at  (17.6,3.5){$3$};
\node at  (11,3){$\bullet$};\node at  (11.6,3.5){$5$};
\node at  (14,0){$\bullet$};\node at  (14.6,0.5){$6$};
\node at  (11,-3){$\bullet$};\node at  (11,-4){$10$};
\node at  (8,0){$\bullet$};\node at  (7.3,0.3){$7$};
\node at  (5,-3){$\bullet$};\node at  (4.3,-2.7){$8$};
\node at  (8,-6){$\bullet$};\node at  (8.5,-5.3){$9$};

\node at  (19.1,1.5){$\rightarrow$};\node at  (19.1,2.5){$\phi$};

\draw[-](27,9) to (21,3);
\draw[-](27,9) to (30,6);\draw[-](24,0) to (30,6);
\draw[-](24,0) to (27,-3);\draw[-](27,3) to (33,-3);\draw[-](30,-6) to (33,-3);
\node at  (30,0){$\bullet$};\node at  (30.6,0.3){$7$};
\node at  (33,-3){$\bullet$};\node at  (33.6,-2.7){$8$};
\node at  (30,-6){$\bullet$};\node at  (30,-5){$9$};
\node at  (27,9){$\bullet$};\node at  (26.3,9.3){$1$};

\node at  (24,6){$\bullet$};\node at  (23.3,6.3){$2$};
\node at  (21,3){$\bullet$};\node at  (21.7,2.8){$3$};
\node at  (30,6){$\bullet$};\node at  (30.6,6.3){$4$};
\node at  (27,3){$\bullet$};\node at  (26.4,3.2){$5$};
\node at  (24,0){$\bullet$};\node at  (23.4,0.2){$6$};
\node at  (27,-3){$\bullet$};\node at  (27,-4){$10$};

\node at  (33.2,2){$\longrightarrow$};\node at  (33.3,3){$\varphi^{-1}$};

\node at  (44,9){$\bullet$};\node at  (43.2,9.2){$0$};
\draw[-](44,9) to (36,0);\draw[-](44,9) to (47,6);\draw[-](47,6) to (44,3);
\node at  (41.33,6){$\bullet$};\node at  (40.6,6.3){$1$};
\node at  (38.66,3){$\bullet$};\node at  (38,3.3){$2$};
\node at  (36,0){$\bullet$};\node at  (35.4,0.3){$3$};
\draw[-](47,6) to (50,3);\draw[-](50,3) to (50,0);\draw[-](47,6) to (47,3);
\node at  (47,6){$\bullet$};\node at  (47,7){$4$};
\node at  (50,3){$\bullet$};\node at  (50,4){$8$};
\node at  (50,0){$\bullet$};\node at  (50,-1){$9$};
\node at  (47,3){$\bullet$};\node at  (47,2){$7$};
\node at  (44,3){$\bullet$};\node at  (44,4){$5$};
\node at  (42,0){$\bullet$};\node at  (42,-1){$6$};
\node at  (46,0){$\bullet$};\node at  (46,-1){$10$};
\draw[-](44,3) to (42,0);\draw[-](44,3) to (46,0);
\end{tikzpicture}
\caption{An example of the involution $\tilde\phi=\varphi^{-1}\circ\phi\circ\varphi$.\label{invo1}}
\end{figure}

For a node $v$ of a plane tree $T\in\P_n$, denote by $\lev_T(v)$ the level of $v$ in $T$ and by $\lsw_T(v)$ the {\em left spanning width\footnote{The notion of left spanning widths was also noticed by Bai and Chen in their revised version of~\cite{BC}.} of $v$}, which is the number of children of $v$ plus the number of edges attached to other nodes (than $v$) on the path from $v$ to the root  from the left-hand side. As long as there is no danger of confusion, we can remove the subscript $T$ from $\lev_T(v)$ and $\lsw_T(v)$.  Note that the distributions of left spanning widths and the right spanning widths are symmetric over plane trees, but are different in general  over  weakly increasing trees introduced in Definition~\ref{def:wit}.

Let $T\in\P_n$ and let $B=\varphi(T)\in\B_n$. 
We supply nodes to $B$ to get a complete binary tree $\overbar{B}$ such that each node in $B$ has exactly two children. We label the supplied nodes in  $\overbar{B}$ by using $0',1',\ldots,n'$ according to the following rule:
\begin{itemize}
\item if the supplied node is a left child of a node $v$ in $B$, then it is labeled by $v'$;
\item if the supplied node is a right child of a node $v$ in $B$, then it is labeled by $u'$, where $u$ is the father of $v$ in $T$. 
\end{itemize}
Then, such a $\overbar B$ is called the {\em complement of $B$}. For instance, if $T$ and $B$ are respectively the first two trees in Fig.~\ref{invo1}, then $\overbar B$ is the tree drawn in Fig.~\ref{bin:comp}. For any node $v$ of $T$, there is a unique node $u$ in $T$ such that  either $v'$ is the right child of $u$ in $\overbar B$ or $u'$ is the right child of $v$ in $\overbar{\phi(B)}$. We write $v\leadsto u$ for the relation of such two nodes. For the nodes of the first plane tree in Fig.~\ref{invo1}, we have 
\begin{align*}
0\leadsto 3, 1\leadsto 6, 2\leadsto1,3\leadsto2, 4\leadsto0, 5\leadsto7, 6\leadsto10,7\leadsto9, 8\leadsto4, 9\leadsto8, 10\leadsto5.
\end{align*}
It is possible to describe the relation $v\leadsto u$ directly using $T$; see Lemma~\ref{dirc}. 
The reason to create the complements of binary trees lies in the following key observation. 

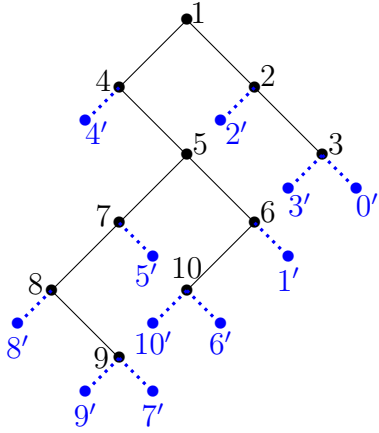
\begin{figure}
\centering
\begin{tikzpicture}[scale=0.3]

\draw[-](11,9) to (8,6);
\draw[-](11,9) to (17,3);\draw[-](8,6) to (14,0);
\draw[-](11,-3) to (14,0);\draw[-](11,3) to (5,-3);\draw[-](8,-6) to (5,-3);
\node at  (11,9){$\bullet$};\node at  (11.5,9.3){$1$};
\node at  (8,6){$\bullet$};\node at  (7.3,6.3){$4$};
\draw[very thick,dotted,blue](8,6) to (6.5,4.5);\node at  (6.5,4.5){\blue{$\bullet$}};
\node at  (7,4){\blue{$4'$}};
\node at  (14,6){$\bullet$};\node at  (14.6,6.5){$2$};\draw[very thick,dotted,blue](14,6) to (12.7,4.7);
\node at  (13.2,4){\blue{$2'$}};\node at  (12.5,4.5){\blue{$\bullet$}};
\node at  (17,3){$\bullet$};\node at  (17.6,3.5){$3$};\draw[very thick,dotted,blue](17,3) to (15.5,1.5);
\node at  (15.5,1.5){\blue{$\bullet$}};\node at  (16,0.8){\blue{$3'$}};
\draw[very thick,dotted,blue](17,3) to (18.5,1.5);\node at  (18.5,1.5){\blue{$\bullet$}};\node at  (19,0.8){\blue{$0'$}};
\node at  (11,3){$\bullet$};\node at  (11.6,3.5){$5$};
\node at  (14,0){$\bullet$};\node at  (14.6,0.5){$6$};\draw[very thick,dotted,blue](14,0) to (15.5,-1.5);
\node at  (15.5,-1.5){\blue{$\bullet$}};\node at  (15.5,-2.5){\blue{$1'$}};
\node at  (11,-3){$\bullet$};\node at  (11,-2){$10$};\draw[very thick,dotted,blue](11,-3) to (9.5,-4.5);
\draw[very thick,dotted,blue](11,-3) to (12.5,-4.5);
\node at  (12.5,-4.5){\blue{$\bullet$}};\node at  (12.5,-5.3){\blue{$6'$}};
\node at  (9.5,-4.5){\blue{$\bullet$}};\node at  (9.5,-5.3){\blue{$10'$}};

\node at  (8,0){$\bullet$};\node at  (7.3,0.3){$7$};\draw[very thick,dotted,blue](8,0) to (9.5,-1.5);
\node at  (9.5,-1.5){\blue{$\bullet$}};\node at  (9.2,-2.3){\blue{$5'$}};
\node at  (5,-3){$\bullet$};\node at  (4.3,-2.7){$8$};\draw[very thick,dotted,blue](5,-3) to (3.5,-4.5);
\node at  (3.5,-4.5){\blue{$\bullet$}};\node at  (3.5,-5.5){\blue{$8'$}};
\node at  (8,-6){$\bullet$};\node at  (7.2,-6){$9$};\draw[very thick,dotted,blue](8,-6) to (6.5,-7.5);
\draw[very thick,dotted,blue](8,-6) to (9.5,-7.5);
\node at  (6.5,-7.5){\blue{$\bullet$}};\node at  (6.5,-8.5){\blue{$9'$}};
\node at  (9.5,-7.5){\blue{$\bullet$}};\node at  (9.5,-8.5){\blue{$7'$}};

\end{tikzpicture}
\caption{The complement of a binary tree.\label{bin:comp}}
\end{figure}

\begin{lemma}\label{lev:lsw}
Fix $T\in\P_n$ and let $B=\varphi(T)\in\B_n$. For any node $v$ of $T$, $\lev(v)$ (resp.~$\lsw(v)$) equals the number of left (resp.~right) edges in the path from the root $1$ to $v'$ in $\overbar B$. 
\end{lemma}
\begin{proof}
This observation follows from  $\varphi$ and the construction of the complement of binary trees. 
\end{proof}

For a node $v$ of a plane tree $T\in\P_n$, let $\deg_T(v)$ be the number of children of $v$  in $T$. We also introduce the {\em dual degree of $v$}, denoted $\widetilde\deg_T(v)$, according to the following three cases:
\begin{itemize}
\item if $v$ is an internal node, then set $\widetilde\deg_T(v)=0$;
\item if $v$ is a leaf in $T$ and all nodes in the path from $v$ to the root of $T$ have no elder siblings, then set $\widetilde\deg_T(v)=\lev_T(v)$;
\item  otherwise, let $\widetilde\deg_T(v)$ be the number of nodes from $v$ to the first node that has elder siblings when walking along the path from $v$ to the root of $T$. 
\end{itemize}
Note that $\widetilde\deg_T(v)$ is positive iff $v$ is a leaf. 
For example, if $T$ is the last tree in  Fig.~\ref{invo1}, then  $\widetilde\deg_T(3)=\widetilde\deg_T(6)=3, \widetilde\deg_T(10)=\widetilde\deg_T(7)=1$ and $\widetilde\deg_T(9)=2$.
By Lemma~\ref{lev:lsw} and the mirror symmetry of binary trees, we have the following generalization of Theorem~\ref{H:RSW}. 
\begin{theorem}\label{thm:lsw}
Fix $T\in\P_n$. For any node $v$ of $T$, if $v\leadsto u$, then 
$$
\lev_T(v)=\lsw_{\tilde\phi(T)}(u)\quad\text{and}\quad \lsw_T(v)=\lev_{\tilde\phi(T)}(u).
$$
Moreover, $\deg_T(v)=\widetilde\deg_{\tilde\phi(T)}(u)$. 
\end{theorem}

\subsection{Generalized to weakly increasing trees}
Recall that an {\em increasing tree} with $n$ edges is a plane tree with $n+1$ nodes labeled by $\{0,1,\ldots,n\}$ such that each child receives greater label than its parent and siblings are labeled increasingly from left to right. Increasing trees are in bijection with permutations and has been extensively studied from the enumerative aspect; see~\cite{Chen,KPP} and related references therein. As a unification of increasing trees and plane trees, the weakly increasing trees labeled by a multiset were introduced by Lin--Ma--Ma--Zhou~\cite{Lin20} in 2021. Since then, various intriguing connections and bijections for weakly increasing trees have already been found~\cite{Lin20,LM,CFu,LLWZ}.  

\begin{definition}[Weakly increasing trees~\cite{Lin20}]\label{def:wit}
Fix a multiset $M=\{1^{p_1},2^{p_2},\cdots,n^{p_n}\}$  with $p=p_1+\cdots+p_n$. A
{\em weakly increasing tree} on $M$ is a plane tree with $p+1$ nodes that are labeled precisely by all elements in the multiset $M\cup\{0\}$ satisfying
\begin{enumerate}[(i)]
\item  the labels along a path from the root to any leaf are weakly increasing (vertical weakly increasing);
\item  the labels of the children of each node is weakly increasing from left to right (horizontal weakly increasing). 
\end{enumerate} 
 Denote by $\mathcal{P}_{M}$ the set of weakly increasing  trees on $M$.   
\end{definition}
Note that weakly increasing trees on $[n]:=\{1,2,\ldots,n\}$ are exactly  increasing trees with $n$ edges, while weakly increasing trees on $\{1^n\}$ are in obvious bijection with plane trees of $n$ edges.
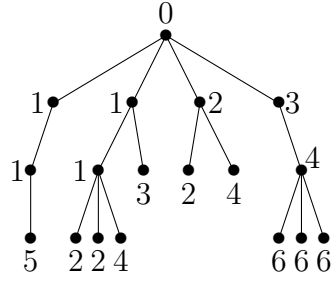
\begin{figure}
\centering
\begin{tikzpicture}[scale=0.3]
\node at (20,20) {$\bullet$};
\node at (15,17) {$\bullet$};
\node at (18.5,17) {$\bullet$};
\node at (21.5,17){$\bullet$};
\node at (21,14){$\bullet$};
\node at (23,14){$\bullet$};

\node at (25,17) {$\bullet$};
\node at (17,14) {$\bullet$};
\node at (16,11) {$\bullet$};
\node at (17,11) {$\bullet$};
\node at (18,11) {$\bullet$};

\node at (19,14) {$\bullet$};
\node at (14,14){$\bullet$};
\node at (14,11) {$\bullet$};
\node at  (26,14) {$\bullet$};

\draw[-] (26,14) to (25,11);\draw[-] (26,14) to (27,11);\draw[-] (26,14) to (26,11);
\node at  (25,11) {$\bullet$};\node at  (27,11) {$\bullet$};\node at  (26,11) {$\bullet$};
\node at (25,10) {$6$};\node at (27,10) {$6$};\node at (26,10) {$6$};

\node at (20,21) {$0$};
\node at (14.3,17) {$1$};
\node at (17.8,17) {$1$};
\node at (22.2,17){$2$};
\node at (25.6,17) {$3$};
\node at (13.4,14) {$1$};
\node at (14,10) {$5$};
\node at (16.2,14) {$1$};
\node at (19,13) {$3$};
\node at (21,13) {$2$};
\node at (23,13) {$4$};
\node at (26.5,14.6) {$4$};
\node at (16,10) {$2$};
\node at (17,10) {$2$};
\node at (18,10) {$4$};

\draw[-] (25,17) to (26,14);
\draw[-] (20,20) to (15,17);
\draw[-] (20,20) to (18.5,17);
\draw[-] (20,20) to (21.5,17);
\draw[-] (20,20) to (25,17);
\draw[-] (15,17) to (14,14);
\draw[-] (14,14) to (14,11);
\draw[-] (18.5,17) to (17,14);
\draw[-] (18.5,17) to (19,14);
\draw[-] (17,14) to (16,11);
\draw[-] (17,14) to (17,11);
\draw[-] (17,14) to (18,11);
\draw[-] (21.5,17) to (21,14);
\draw[-] (21.5,17) to (23,14);
\end{tikzpicture}
\caption{A weakly increasing tree on $\{1^4,2^4,3^2,4^3,5,6^3\}$.\label{wit}}
\end{figure}
See Fig.~\ref{wit} for a weakly increasing tree on $M=\{1^4,2^4,3^2,4^3,5,6^3\}$. 
It turns out that the bijection $\varphi$ between plane trees and binary trees can be extended naturally to a bijection between  weakly increasing trees and weakly increasing binary trees defined below. 

 \begin{definition}[Weakly increasing binary trees~\cite{LM}]
 Fix a multiset $M$. A  {\em weakly increasing binary  tree} on $M$ is a labeled  binary tree  such that 
 \begin{enumerate}[(i)]
\item  the labels of the nodes form precisely the multiset $M$ and
\item  the labels along a path from the root to any leaf is weakly increasing. 
\end{enumerate}   
 Denote by $\mathcal{B}_{M}$ the set of weakly increasing binary trees on $M$. 
\end{definition}
Note that weakly increasing binary trees on $[n]$ are exactly  increasing binary trees on $[n]$, while weakly increasing binary trees on $\{1^n\}$ are in obvious bijection with binary trees of $n$ nodes.
Since the mirror symmetry of a weakly  increasing binary tree is still a weakly increasing binary tree, the involution $\tilde\phi$ introduced for plane trees can be extended directly to weakly increasing trees $\tilde\phi:\P_M\rightarrow\P_M$. Then the same discussions as in the proof of Theorem~\ref{thm:lsw} for plane trees works for weakly increasing trees. Moreover, it turns out that such involution $\tilde\phi$ on $\P_M$ also reproves a quintuple equidistribution on weakly increasing trees found by Chen and Fu~\cite{CFu}.

To state our result, we need to recall some statistics on weakly increasing trees. Let $T$ be a weakly increasing tree. Let $\leaf(T)$ (resp.~$\intt(T)$) be the number of leaves (resp.~internal nodes) in $T$. The levels (resp.~left spanning widths) of all nodes in $T$ forms a multiset that is denoted as $\Lev(T)$ (resp.~$\Lsw(T)$). Each node that has an elder sibling (resp.~a child) with the same label is called a {\em repeated sibling (resp.~repeated parent)}. The labels of all repeated siblings (resp.~parents) in $T$ form a multiset that is denoted as $\Bro(T)$ (resp.~$\Par(T)$). A leaf  of a weakly increasing tree is {\em young} if it is the rightmost child of its parent; otherwise, it is old. Denote $\Yle(T)$ the multiset of labels of all young leaves in $T$. Young and old leaves on plane trees were introduced by Chen, Deutsch and Elizalde~\cite{Chen2006}. Take the tree $T$ in Fig.~\ref{wit} as an example, we have $\leaf(T)=10$, $\intt(T)=8$, $\Lev(T)=\{0,1^4,2^6,3^7\}$, $\Lsw(T)=\{0,1^3,2^3,3^5,4^4,5,6\}$, $\Bro(T)=\{1,2,6,6\}$, $\Par(T)=\{1,1,2\}$ and $\Yle(T)=\{3,4,4,5,6\}$.
Our first main result is the following septuple equidistribution on weakly increasing trees that merges both results of Bai--Chen~\cite{BC} and Chen--Fu~\cite{CFu}. 
\begin{theorem}\label{thm:lswit}
Fix a multiset $M$ a tree $T\in\P_M$. For any node $v$ of $T$, if $v\leadsto u$, then 
$$
\lev_T(v)=\lsw_{\tilde\phi(T)}(u),\quad \lsw_T(v)=\lev_{\tilde\phi(T)}(u) \quad\text{and}\quad \deg_T(v)=\widetilde\deg_{\tilde\phi(T)}(u).
$$
Consequently, the following two septuples
\begin{equation}\label{eq:sept}
(\Lev,\Lsw,\Bro,\Par,\Yle,\leaf,\intt)\quad\text{and}\quad(\Lsw,\Lev,\Par,\Bro,\Yle,\intt,\leaf)
\end{equation}
have the same distribution over $\P_M$.
\end{theorem}

\begin{proof}
The first statement follows from the same discussions as in the proof of Theorem~\ref{thm:lsw} for plane trees. For the second statement, notice that a  repeated sibling (resp.~a young leaf) maps to a repeated right child (resp.~a leaf) under $\varphi$, then to a repeated left child (resp.~a leaf) under $\phi$, and finally to a repeated parent (resp.~a young leaf) under  $\varphi^{-1}$.
\end{proof}

\begin{remark}
The septuple equidistribution~\eqref{eq:sept} that does not involve $\Lev$ and $\Lsw$ was proved by Chen and Fu~\cite{CFu} using a recursively defined involution. 
\end{remark}

\subsection{Some new statistics having the same distribution as levels on plane trees}

Since the statistics $\lsw$ and $\rsw$ are in symmetry with each other, Bai and Chen  posed the following open problem in their revised version of~\cite{BC}. 
\begin{problem}[Bai and Chen]\label{prob:bij2}
	Does there exist another statistic $\sat$ associated to nodes of plane trees such that the
	pair $(\lsw, \rsw)$ is in equidistribution with the pair $(\lev,\sat)$?
\end{problem}

They believed  that such a statistic exists but could not find it. 
We will answer this problem affirmatively and introduce four new statistics  analogous to $\lsw$ and  $\rsw$, having the same distribution as levels on plane trees. 

 Let $v$ be a node of a plane tree $T$. Let $\Lc_T(v)$ be the set of nodes attached to other nodes (than $v$) on the path from $v$ to the root from the left-hand side. For example, if $T$ is the first tree in Fig.~\ref{lchain}, then $\Lc_T(10)=\{4,5\}$ and $\Lc_T(9)=\{4,8\}$. Let $\lc_T(v)=|\Lc_T(v)|$. The statistic `$\lc$' is a refinement of `$\lsw$' in the sense that $\lsw_T(v)=\deg_T(v)+\lc_T(v)$.
Our first new statistic is a modification of `$\lc$' defined as 
$$
\dlev_T(v)=
\begin{cases}
0, &\text{if $v$ is the root of $T$};\\
\lc_T(v)+1, &\text{otherwise}.
\end{cases}
$$

	The following observation follows immediately from   $\varphi$. 
\begin{lemma}\label{lev:rlsw}
	Fix $T\in\P_n$ and let $B=\varphi(T)\in\B_n$. For any node $v$ of $T$ other than the root, $\lev_{T}(v)-1$ (resp.,~$\lc_{T}(v)$) equals the number of left (resp.,~right) edges in the path from the root $1$ to $v$ in $B$. 
\end{lemma}

This simple lemma has one interesting application about longest increasing/decreasing subsequences in $132$-avoiding permutations. Recall that a permutation $\pi=\pi_1\pi_2\cdots\pi_n$ is said to be {\em $132$-avoiding} if there are no indices $1\leq i<j<k\leq n$ such that $\pi_i<\pi_k<\pi_j$. Let $\SS_n(132)$ be the set of all $132$-avoiding permutations of length $n$. For each permutation $\pi$, denote $\is(\pi)$ (resp.~$\ds(\pi)$) the length of the longest increasing (resp.~decreasing) subsequence of $\pi$. The distributions of these two statistics over $132$-avoiding permutations were studied by Krattenthaler~\cite{Kra} and Reifegerste~\cite{Re}. Combining $\tilde\phi$ with Jani and Rieper's bijection~\cite{JR} between $\P_n$ and $\SS_n(132)$ leads to the following inequality distribution.

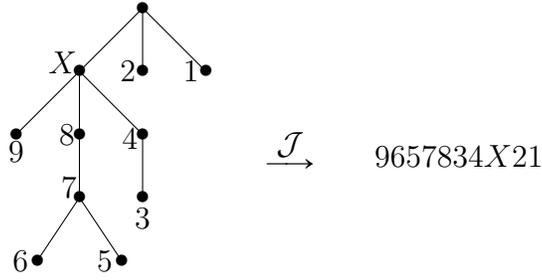
\begin{figure}
	\centering
	\begin{tikzpicture}[scale=0.28]
	\draw[-](0,9) to (-3,6);\draw[-](0,9) to (0,6);\draw[-](0,9) to (3,6);
		\draw[-](-3,6) to (-6,3);\draw[-](-3,6) to (0,3);\draw[-](0,3) to (0,0);
		\draw[-](-3,6) to (-3,3);\draw[-](-3,0) to (-3,3);
		\draw[-](-3,0) to (-5,-3);\draw[-](-3,0) to (-1,-3);
		\node at  (0,9){$\bullet$};
		\node at  (-3,6){$\bullet$};\node at  (-3.8,6.4){$X$};
		\node at  (0,6){$\bullet$};\node at  (-0.7,6){$2$};
		\node at  (3,6){$\bullet$};\node at  (2.3,6){$1$};
		\node at  (-6,3){$\bullet$};\node at  (-6,2.2){$9$};
		\node at  (0,3){$\bullet$};\node at  (-0.6,2.8){$4$};
		\node at  (-3,3){$\bullet$};\node at  (-3.6,3){$8$};
		\node at  (-3,0){$\bullet$};
		\node at  (-5,-3){$\bullet$};\node at  (-5.8,-3){$6$};
		\node at  (-1,-3){$\bullet$};\node at  (-1.8,-3){$5$};
		\node at  (0,0){$\bullet$};\node at  (0,-1){$3$};
		\node at  (-3.5,0.5){$7$};
		
		\node at  (7,1.5){$\longrightarrow$};\node at  (7,2.5){$\JR$};
		
		\node at  (15,2){$9657834X21$};
	
	\end{tikzpicture}
	\caption{An example of Jani--Rieper's bijection $\JR$.\label{jani}}
\end{figure}
\begin{proposition}\label{is:ds}
There exists an involution $\theta$ on $\SS_n(132)$ such that 
\begin{equation}
\is(\pi)\leq \ds(\theta(\pi))\quad\text{and}\quad\is(\theta(\pi))\leq \ds(\pi)
\end{equation}
for each $\pi\in\SS_n(132)$. In particular, we have 
$$
\sum_{\pi\in\SS_n(132)} \is(\pi)\leq \sum_{\pi\in\SS_n(132)} \ds(\pi). 
$$
\end{proposition}

Let us first review briefly Jani and Rieper's bijection $\JR:\P_n\rightarrow\SS_n(132)$. 
Given a plane tree $T\in\P_n$, a node of $T$ is said to be the $i$-th node under the preorder (resp.~reverse preorder) in $T$ iff it is the 
$i$-th node under the preorder (resp.~reverse preorder) in $\varphi(T)$.  Note that the preorder of plane trees defined here agrees with that in~\cite[Page~10]{St1} and a node $v$ is the $i$-th node under the preorder iff it is the $(n+2-i)$-th node under the reverse preorder. See Fig.~\ref{plan:preorder} for a plane tree $T$ labeled by its preorder. The permutation $\JR(T)\in\SS_n(132)$ is constructed in two steps:
\begin{itemize}
\item label each non-root node of $T$ by its reverse preorder;
\item then $\JR(T)$ is the word obtained by  reading the labeled tree in {\em postorder}, that is, recursively traversing the left subtree  to the right subtrees then to the parent.
\end{itemize}
See Fig.~\ref{jani} for an example of Jani and Rieper's bijection $\JR$.

\begin{remark}
The inverse of $\JR$ is essentially the restriction of the classical bijection (see~\cite[Example~1.3.15]{St0}) between permutations and increasing trees on $213$-avoiding permutations. In fact, given $\pi\in\SS_n(132)$, for $i$ from $n$ down to $1$, construct the (labeled) plane tree $\JR^{-1}(\pi)$ by defining $\pi_i$ to be the child of the leftmost letter  $\pi_j$ ($j>i$) which is greater than $\pi_i$. If there is no such letter $\pi_j$, then let $\pi_i$ be the child of the root. 
\end{remark}

\begin{proof}[{\bf Proof of Proposition~\ref{is:ds}}] We set $\theta=\JR\circ\tilde{\phi}\circ\JR^{-1}$. Since $\theta$ is an involution, we only need to show $\is(\pi)\leq \ds(\theta(\pi))$.

For $\pi\in\SS_n(132)$ and an index $i\in[n]$, introduce 
\begin{align*}
\IR(\pi_i)&=\{\pi_j: j\geq i\text{ and }\pi_j\geq\pi_{\ell}\text{ for all $i\leq\ell<j$}\},\\
\IL(\pi_i)&=\{\pi_j: j\leq i\text{ and }\pi_j\geq\pi_{\ell}\text{ for all $j<\ell\leq i$}\}.
\end{align*}
Suppose that  $T=\JR^{-1}(\pi)$ whose non-root nodes are labeled as in the first step of $\JR$, i.e., according to the reverse preorder. Then, we have the following two observations:
\begin{enumerate}
\item $\IR(\pi_i)$ equals the set of nodes (other than the root) in the path from $\pi_i$ to the root;
\item $\IL(\pi_i)=\Lc_T(\pi_i)\cup\{\pi_i\}$. 
\end{enumerate}
If we introduce the two multisets 
$$
\IR(\pi)=\{|\IR(\pi_i)|:i\in[n]\}\quad\text{and}\quad \IL(\pi)=\{|\IL(\pi_i)|:i\in[n]\},
$$ 
then by Lemma~\ref{lev:rlsw}, we have $\IR(\pi)=\IL(\theta(\pi))$. As $\pi$ is $132$-avoiding, $\is(\pi)=\max(\IR(\pi))$. On the other hand, $\ds(\sigma)\geq \max(\IL(\sigma))$ for any permutation $\sigma$.
This leads to the conclusion that $\is(\pi)\leq \ds(\theta(\pi))$. 
\end{proof} 

Before introducing the other three statistics, we need an auxiliary definition. The {\em left chain} of a node $v$ in $T$ is a path $v=v_{1},v_{2},...,v_{k}$, where $v_{i+1}$ is the eldest child of $v_{i}$ for $i=1,\ldots,k-1$ and $v_{k}$ is a leaf in $T$. 
Denote by $\lchain_{T}(v)$ the length of the left chain of $v$ in $T$. 
Our second statistic is  defined as
$$
\dlsw_{T}(v)=
\begin{cases}
\lchain_{T}(v), &\text{if $v$ is the root};\\
\lchain_{T}(u)+\lev_{T}(v), &\text{if $v$ has the closest younger sibling  $u$};\\
\lev_{T}(v)-1, &\text{otherwise}.
\end{cases}
$$
Our third statistic is  defined as 
$$
\drsw_{T}(v)=
\begin{cases}
1+\sum\limits_{w\in\Lc_{T}(v)\cup\{u,v\}}\lchain_{T}(w), &\text{if $v$ has the closest younger sibling  $u$};\\
\sum\limits_{w\in\Lc_{T}(v)\cup\{v\}}\lchain_{T}(w), &\text{otherwise}.
\end{cases}
$$

\begin{theorem}\label{thm:new three}
	For any plane tree $T$ and a node $v$ of $T$, we have
	\begin{equation}\label{Statistics:two}
		(\lev,\lsw,\rsw)_{T}(v)=(\dlev,\dlsw,\drsw)_{\tilde\phi(T)}(v).
	\end{equation}
\end{theorem}

In order to prove Theorem~\ref{thm:new three}, we  need  a direct description of the  involution $\tilde{\phi}$ as follows:
	  \begin{enumerate}[(i)]
	  \item the node $1$ is the eldest child of the root $0$ in $\tilde{\phi}(T)$;
	  	\item  if $v$ is the eldest child of a node $u$ other than the root in $T$, then $v$ becomes the closest younger sibling of $u$ in $\tilde{\phi}(T)$;
	  	\item  if $v$ is not the eldest child of its parent and $u$ is the closest  sibling elder  than $v$ in $T$, then $v$ is the eldest child of $u$ in $\tilde{\phi}(T)$.
	  \end{enumerate} 
See Fig.~\ref{lchain} for an example of the direct description of $\tilde{\phi}$. The following observation is clear from this direct description of $\tilde{\phi}$.

\begin{lemma}\label{lsw:dlsw}
	Fix $T\in\P_n$ and let $\tilde{T}=\tilde{\phi}(T)$. For any internal node $v$ of $T$ other than the root, if $u$ is the eldest child of $v$ in $T$, then 
	$$
	\deg_{T}(v)=\lchain_{\tilde{T}}(u)+1.
	$$
\end{lemma}

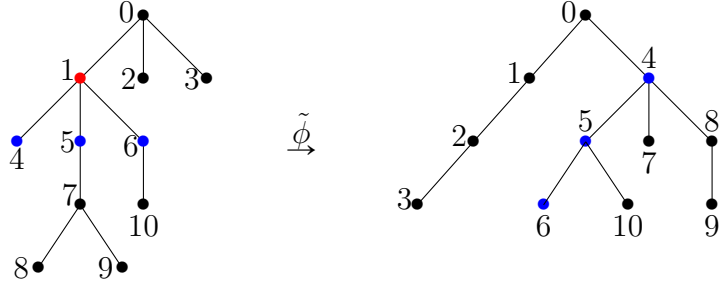
\begin{figure}
	\centering
	
	\begin{tikzpicture}[scale=0.28]
		\draw[-](0,9) to (-3,6);\draw[-](0,9) to (0,6);\draw[-](0,9) to (3,6);
		\draw[-](-3,6) to (-6,3);\draw[-](-3,6) to (0,3);\draw[-](0,3) to (0,0);
		\draw[-](-3,6) to (-3,3);\draw[-](-3,0) to (-3,3);
		\draw[-](-3,0) to (-5,-3);\draw[-](-3,0) to (-1,-3);
		\node at  (0,9){$\bullet$};\node at  (-0.8,9.2){$0$};
		\node[red] at  (-3,6){$\bullet$};\node at  (-3.6,6.4){$1$};
		\node at  (0,6){$\bullet$};\node at  (-0.7,6){$2$};
		\node at  (3,6){$\bullet$};\node at  (2.3,6){$3$};
		\node[blue] at  (-6,3){$\bullet$};\node at  (-6,2.2){$4$};
		\node[blue] at  (0,3){$\bullet$};\node at  (-0.6,2.8){$6$};
		\node[blue] at  (-3,3){$\bullet$};\node at  (-3.6,3){$5$};
		\node at  (-3,0){$\bullet$};
		\node at  (-5,-3){$\bullet$};\node at  (-5.8,-3){$8$};
		\node at  (-1,-3){$\bullet$};\node at  (-1.8,-3){$9$};
		\node at  (0,0){$\bullet$};\node at  (0,-1){$10$};
		\node at  (-3.5,0.5){$7$};
		
		\node at  (7.5,2.5){$\rightarrow$};\node at  (7.5,3.5){$\tilde{\phi}$};
		
		\node at  (21,9){$\bullet$};\node at  (20.2,9.2){$0$};
		\draw[-](21,9) to (13,0);\draw[-](21,9) to (24,6);\draw[-](24,6) to (21,3);
		\node at  (18.33,6){$\bullet$};\node at  (17.6,6.3){$1$};
		\node at  (15.66,3){$\bullet$};\node at  (15,3.3){$2$};
		\node at  (13,0){$\bullet$};\node at  (12.4,0.3){$3$};
		\draw[-](24,6) to (27,3);\draw[-](27,3) to (27,0);\draw[-](24,6) to (24,3);
		\node[blue] at  (24,6){$\bullet$};\node at  (24,7){$4$};
		\node at  (27,3){$\bullet$};\node at  (27,4){$8$};
		\node at  (27,0){$\bullet$};\node at  (27,-1){$9$};
		\node at  (24,3){$\bullet$};\node at  (24,2){$7$};
		\node[blue] at  (21,3){$\bullet$};\node at  (21,4){$5$};
		\node[blue] at  (19,0){$\bullet$};\node at  (19,-1){$6$};
		\node at  (23,0){$\bullet$};\node at  (23,-1){$10$};
		\draw[-](21,3) to (19,0);\draw[-](21,3) to (23,0);
	\end{tikzpicture}
	\caption{An example of $\deg_{T}(1)=\lchain_{\tilde{T}}(4)+1$.\label{lchain}}
\end{figure}

\begin{lemma}\label{rsw:drsw}
	Fix $T\in\P_n$ and let $\tilde{T}=\tilde{\phi}(T)$. For any node $v$ of $T$, we have
	$$
	\rsw_{T}(v)=\drsw_{\tilde{T}}(v).
	$$
\end{lemma}
\begin{proof}
	Suppose the path from $v$ to the root in $T$ is $v_{0},v_{1},...,v_{k}$ with $v_k=0$. By definition, $\rsw_T(v)$ equals the sum of the numbers of all the younger siblings of $v_{i}$ ($0\leq i<k$) plus $\deg_{T}(v)$. 
	By the direct description of $\tilde{\phi}$, the nodes $v_{1},...,v_{k-1}$  of $T$ are mapped to the nodes in $\Lc_{\tilde{T}}(v)$ and all the younger siblings of $v_{i}$ ($0\leq i<k$)  in $T$ are mapped to the rest nodes of left chain of $v_{i}$ other than $v_{i}$ in $\tilde{T}$. Combining with Lemma~\ref{lsw:dlsw} finishes the proof of the lemma.
\end{proof}

Now, we are in position to prove Theorem~\ref{thm:new three}. 
\begin{proof}[{\bf Proof of Theorem~\ref{thm:new three}}]
	By Lemmas~\ref{lev:rlsw} and~\ref{rsw:drsw}, we have 
	$$
	(\lev,\rsw)_{T}(v)=(\dlev,\drsw)_{\tilde\phi(T)}(v).
	$$ 
	It remains to show that $\lsw_T(v)=\dlsw_{\tilde\phi(T)}(v)$. We distinguish the following three cases.
	
	\begin{enumerate}
	\item If $v=0$ is the root  of $T$, then it follows from the direct description of $\tilde{\phi}$ that $\deg_T(0)=\lchain_{\tilde\phi(T)}(0)$ and so $\lsw_T(0)=\dlsw_{\tilde\phi(T)}(0)$.
	\item If $v$ is a leaf of $T$, then $\lsw_T(v)=\dlsw_{\tilde\phi(T)}(v)$ by Lemma~\ref{lev:rlsw}. 
	\item Otherwise, $v$ is an internal node  of $T$ other than the root.	By Lemmas~\ref{lev:rlsw} and~\ref{lsw:dlsw}, we have 
	$$
	\lsw_T(v)=\lc_T(v)+\deg_{T}(v)=\lev_{\tilde\phi(T)}(v)-1+\lchain_{\tilde\phi(T)}(u)+1=\dlsw_{\tilde\phi(T)}(v),
	$$
	as desired.
	\end{enumerate}
	The proof of Theorem~\ref{thm:new three} is complete. 
\end{proof}

We will solve Problem~\ref{prob:bij2} by introducing the fourth statistic which is named {\em dual right spanning width of $v$} in $T$: 
$$\widetilde\rsw_{T}(v)=\sum_{w\in\Lc_{T}(v)\cup\mathsf{Ch}_T(v)}\lchain_{T}(w)+\widetilde\deg_{T}(v),$$
where $\mathsf{Ch}_T(v)$ is the set of all children of $v$ in $T$. For example, if $T$ is the second tree in Fig.~\ref{trsw}, then $\widetilde\rsw_{T}(4)=4$ and $\widetilde\rsw_{T}(10)=5$.
The following theorem answers Problem~\ref{prob:bij2}. 

\begin{figure}
	\centering
	\begin{tikzpicture}[scale=0.28]
		\draw[-](0,9) to (-3,6);\draw[very thick,blue](0,9) to (3,6);
		\draw[-](-3,3) to (-6,0);\draw[-](-3,3) to (0,0);\draw[-](0,0) to (-3,-3);
		\draw[-](-3,6) to (-3,3);\draw[very thick,blue](0,0) to (3,-3);\draw[-](-3,-3) to (-6,-6);\draw[very thick,blue](0,-6) to (-3,-3);
		\draw[very thick,blue](-3,-6) to (-3,-3);
		\node at  (0,9){$\bullet$};\node at  (-0.8,9.2){$0$};
		\node at  (-3,6){$\bullet$};\node at  (-3.6,6.4){$1$};
		\node at  (3,6){$\bullet$};\node at  (3,5){$2$};
		\node at  (-6,0){$\bullet$};\node at  (-6,1){$4$};
		\node at  (-3,3){$\bullet$};\node at  (-3.7,3.2){$3$};
		\node at  (-3,-3){$\bullet$};\node at  (-3,-2){$6$};
		\node at  (0,0){$\bullet$};\node at  (0,1){$5$};
		\node at  (3,-3){$\bullet$};\node at  (3,-4){$7$};
		\node at  (0,-6){$\bullet$};\node at  (0,-7){$10$};
		\node[red] at  (-6,-6){$\bullet$};\node at  (-6,-7){$8$};
		\node at  (-3,-6){$\bullet$};\node at  (-3,-7){$9$};
		
		\node at  (8.5,1.5){$\rightarrow$};\node at  (8.5,2.5){$\tilde{\phi}$};
		
		\node at  (17,9){$\bullet$};\node at  (16.2,9.2){$0$};
		\draw[-](17,9) to (14,6);\draw[-](17,9) to (20,6);\draw[-](20,6) to (17,3);\draw[-](17,9) to (17,6);
		\draw[-](14,6) to (14,3);\draw[-](20,3) to (20,0);\draw[-](23,-3) to (23,0);
		\node at  (14,6){$\bullet$};\node at  (13.3,6){$1$};
		\draw[-](20,6) to (23,3);\draw[-](23,3) to (23,0);\draw[-](20,6) to (20,3);
		\node[red] at  (20,6){$\bullet$};\node at  (20,7){$4$};
		\node at  (23,3){$\bullet$};\node at  (23,4){$8$};
		\node[very thick,blue] at  (23,0){$\bullet$};\node at  (24,0){$9$};
		\node at  (20,3){$\bullet$};\node at  (19,3){$6$};
		\node[blue] at  (20,0){$\bullet$};\node at  (20,-1){$7$};
		\node at  (17,3){$\bullet$};\node at  (17,2){$5$};
		\node at  (17,6){$\bullet$};\node at  (17,5){$3$};
		\node[blue] at  (14,3){$\bullet$};\node at  (14,2){$2$};
		\node[blue] at  (23,-3){$\bullet$};\node at  (23,-4){$10$};
	\end{tikzpicture}
	\caption{An example of $\rsw_{T}(8)=\widetilde\rsw_{\tilde{T}}(4)$.\label{trsw}}
\end{figure}
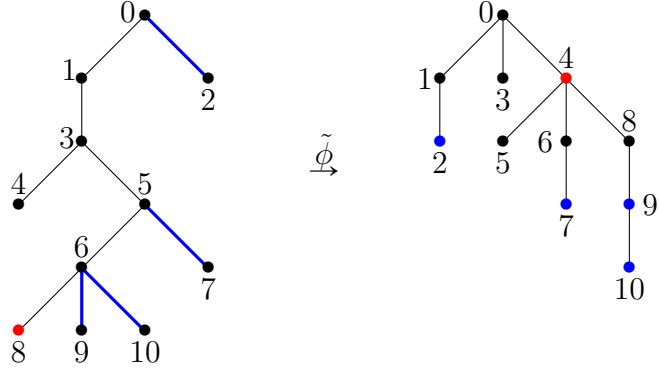

\begin{theorem}\label{thm:dlsw}
	Fix $T\in\P_n$ and let $\tilde{T}=\tilde{\phi}(T)$. For any node $v$ of $T$, if $v\leadsto u$, then 
	\begin{equation}\label{Statistics:four}
		(\lsw,\lev,\rsw,\widetilde\rsw)_{T}(v)=(\lev,\lsw,\widetilde\rsw,\rsw)_{\tilde T}(u).
	\end{equation}
\end{theorem}

Before we prove Theorem~\ref{thm:dlsw}, we need the following lemma, which can be verified from the bijection $\varphi$ and the construction of the complement of binary trees.
\begin{lemma}\label{dirc}
For any node $v$ of a plane tree $T$ such that $v\leadsto u$, the node $u$ can be determined in $T$ according to the following three cases. 
\begin{itemize}
\item If $v$ is an internal node, then $u$ is the youngest child of $v$.
\item If $v$ is a leaf and no nodes in the path from $v$ to the root has elder siblings, then $u$ is the root $0$.
\item If $v$ is a leaf and $w$ is the first node that has elder siblings in the path from $v$ to the root, then $u$ is the closest elder sibling of $w$. 
\end{itemize}
\end{lemma}
We are now ready to prove  Theorem~\ref{thm:dlsw}. 
\begin{proof}[{\bf Proof of Theorem~\ref{thm:dlsw}}]
	As $v\leadsto u$ in $T$ is equivalent to $u\leadsto v$ in $\tilde{T}$ and $\tilde\phi$ is an involution, we only need to show
$\rsw_{T}(v)=\widetilde\rsw_{\tilde{T}}(u)$ in view of Theorem~\ref{thm:lsw}. We distinguish the following three cases.

\begin{enumerate}
\item  If $v$ is an internal node other than the root of $T$, then $u$ is a leaf in $\tilde T$.  By Lemma~\ref{dirc}, $u$ is  the youngest child of $v$ in $T$ and dually, if $u'$ is the first node that has elder siblings in the path from $u$ to the root in $\tilde T$, then $v\in\Lc_{\tilde{T}}(u)$  is the closest elder sibling of $u'$ in $\tilde T$. It then follows that $\widetilde{\deg}_{\tilde{T}}(u)=\lchain_{\tilde{T}}(u')+1$. By Lemma~\ref{rsw:drsw} and the relationship between $u$ and $v$ in $\tilde{T}$, we have
	\begin{align*}
	\rsw_{T}(v) &= \drsw_{\tilde{T}}(v)\\
				&=1+ \sum_{w\in\Lc_{\tilde{T}}(v)\cup\{v,u'\}}\lchain_{\tilde{T}}(w)\\
				&= \sum_{w\in\Lc_{\tilde{T}}(u)}\lchain_{\tilde{T}}(w)+\widetilde{\deg}_{\tilde{T}}(u)\\
				&= \widetilde\rsw_{\tilde{T}}(u).
	\end{align*}
\item If $v$ is the root of $T$, then $\rsw_{T}(v)=\deg_{T}(v)=\widetilde{\deg}_{\tilde{T}}(u)=\widetilde\rsw_{\tilde{T}}(u)$. 
	\item If $v$ is a leaf of $T$, then $v$ is  the youngest child of $u$ in $\tilde{T}$ according to Lemma~\ref{dirc}. By Lemma~\ref{rsw:drsw} and the fact that $u$ is an internal node in $\tilde T$, we have 
		\begin{align*}
		\rsw_{T}(v) &= \drsw_{\tilde{T}}(v)\\
		&= \sum_{w\in\Lc_{\tilde{T}}(v)\cup\{v\}}\lchain_{\tilde{T}}(w)\\
		&= \sum_{w\in\Lc_{\tilde T}(u)\cup\mathsf{Ch}_{\tilde T}(u)}\lchain_{\tilde T}(w)\\
		&= \widetilde\rsw_{\tilde{T}}(u).
	\end{align*}
	\end{enumerate}
The proof of this theorem is complete. 	
\end{proof}

\begin{remark}
The four new tree statistics $(\dlev,\dlsw,\drsw,\widetilde\rsw)$ can be generalized to weakly increasing trees, as well as Theorems~\ref{thm:new three} and~\ref{thm:dlsw}.
\end{remark}

\section{A new involution on binary trees, generalizations and applications}
\label{sec:3}
In this section, we construct a new involution on binary trees, which can be generalized naturally to di-sk trees and rooted labeled trees. The description of this  involution is simple but could not be found in the website of \href{https://www.findstat.org/MapsDatabase/}{FindStat}~\cite{FS}.  This involution answers Problem~\ref{prob:bij} and leads to a new statistic whose distribution over binary trees or plane trees gives the {\em Catalan's triangle~\cite{Cai}} (also known as {\em ballot numbers}) $C_{n,k}=\frac{k}{2n-k}{2n-k\choose n}$. Moreover, a quadruple equidistribution on plane trees involving this new statistic is proved via a recursive bijection. 

For a binary tree (labeled or unlabeled) $B$, let $\spi(B)$ be the  {\em length of the spine} of $B$, which equals the  smallest $i$ such that the $(i+1)$-th node under the {\em preorder} is a right child.  Let $\rspi(B)$ be the smallest $i$ such that the $(i+1)$-th node under the {\em reverse preorder} has a right child. By convention, if $B$ is the special  binary tree with only left edges that has $n$ nodes, then define $\spi(B)=\rspi(B)=n$. As an example, for the first binary tree in  Fig.~\ref{invo2}, we have  $\spi(B)=3$ (counts the three nodes in blue) and $\rspi(B)=4$ (counts the four nodes in red). 

Given a binary tree, its {\em right chain} is any maximal path composed of only right edges. For instance, the first binary tree in Fig.~\ref{invo2} has $9$ right chains, which are $v_1-v_{10}-v_{11}-v_{12}-v_{17}$, $v_2-v_4-v_8$, $v_{13}-v_{15}$, $v_3$, $v_5$, $v_9$, $v_{14}$, $v_{16}$ and $v_6-v_7$. Right chains of binary trees was introduced in the study of multiset Schett polynomials~\cite{LLWZ,LM}. The length of a right chain is the number of right edges in the chain. The multiset of the lengths of all right chains of a binary tree $B$ is called the {\em right chain sequence of $B$}.

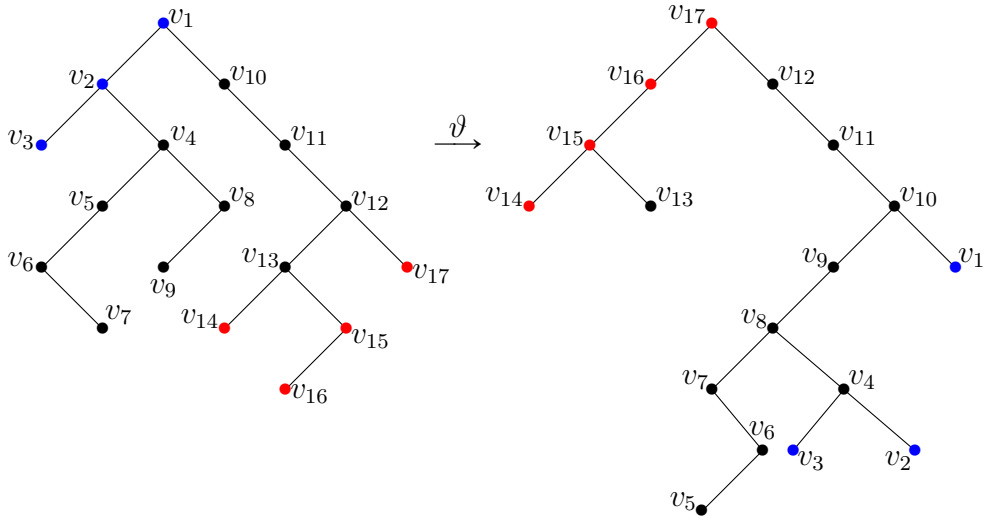
\begin{figure}
\centering
\begin{tikzpicture}[scale=0.27]

\draw[-](11,9) to (8,6);
\draw[-](11,9) to (17,3);\draw[-](8,6) to (14,0);\draw[-](8,6) to (5,3);
\draw[-](11,-3) to (14,0);\draw[-](11,3) to (5,-3);\draw[-](8,-6) to (5,-3);
\draw[-](23,-3) to (17,3);\draw[-](20,0) to (14,-6);\draw[-](17,-3) to (20,-6);
\draw[-](17,-9) to (20,-6);
\node at  (17,-9){\red{$\bullet$}};\node at  (18.2,-9.2){$v_{16}$};
\node at  (20,-6){\red{$\bullet$}};\node at  (21.2,-6.5){$v_{15}$};
\node at  (14,-6){\red{$\bullet$}};\node at  (12.8,-5.5){$v_{14}$};
\node at  (17,-3){$\bullet$};\node at  (15.8,-2.5){$v_{13}$};
\node at  (20,0){$\bullet$};\node at  (21.2,0.3){$v_{12}$};
\node at  (23,-3){\red{$\bullet$}};\node at  (24.2,-3.3){$v_{17}$};
\node at  (5,3){\blue{$\bullet$}};\node at  (4,3.3){$v_3$};
\node at  (11,9){\blue{$\bullet$}};\node at  (11.9,9.3){$v_1$};
\node at  (8,6){\blue{$\bullet$}};\node at  (7,6.3){$v_2$};
\node at  (14,6){$\bullet$};\node at  (15.2,6.5){$v_{10}$};
\node at  (17,3){$\bullet$};\node at  (18.2,3.5){$v_{11}$};
\node at  (11,3){$\bullet$};\node at  (12,3.5){$v_4$};
\node at  (14,0){$\bullet$};\node at  (14.9,0.5){$v_8$};
\node at  (11,-3){$\bullet$};\node at  (11,-4){$v_9$};
\node at  (8,0){$\bullet$};\node at  (7,0.3){$v_5$};
\node at  (5,-3){$\bullet$};\node at  (4,-2.7){$v_6$};
\node at  (8,-6){$\bullet$};\node at  (8.8,-5.3){$v_7$};

\node at  (25.5,3){$\longrightarrow$};\node at  (25.5,4){$\vartheta$};
\draw[-](38,9) to (29,0);\draw[-](32,3) to (35,0);
\draw[-](38,9) to (50,-3);\draw[-](47,0) to (38,-9);
\draw[-](40.5,-12) to (38,-9);\draw[-](41,-6) to (48,-12);\draw[-](44.5,-9) to (42,-12);
\draw[-](40.5,-12) to (37.5,-15);
\node at  (48,-12){\blue{$\bullet$}};\node at  (47.2,-12.5){$v_2$};
\node at  (42,-12){\blue{$\bullet$}};\node at  (42.9,-12.5){$v_3$};
\node at  (44.5,-9){$\bullet$};\node at  (45.4,-8.5){$v_4$};
\node at  (40.5,-12){$\bullet$};\node at  (40.5,-11){$v_6$};
\node at  (37.5,-15){$\bullet$};\node at  (36.6,-14.5){$v_5$};
\node at  (44,-3){$\bullet$};\node at  (43.1,-2.5){$v_9$};
\node at  (41,-6){$\bullet$};\node at  (40.1,-5.5){$v_8$};
\node at  (38,-9){$\bullet$};\node at  (37.2,-8.5){$v_7$};
\node at  (41,6){$\bullet$};\node at  (42.2,6.5){$v_{12}$};
\node at  (44,3){$\bullet$};\node at  (45.2,3.5){$v_{11}$};
\node at  (47,0){$\bullet$};\node at  (48.2,0.5){$v_{10}$};
\node at  (50,-3){\blue{$\bullet$}};\node at  (50.8,-2.5){$v_1$};
\node at  (35,0){$\bullet$};\node at  (36.2,0.5){$v_{13}$};
\node at  (38,9){\red{$\bullet$}};\node at  (36.8,9.5){$v_{17}$};
\node at  (35,6){\red{$\bullet$}};\node at  (33.8,6.5){$v_{16}$};
\node at  (32,3){\red{$\bullet$}};\node at  (30.8,3.5){$v_{15}$};
\node at  (29,0){\red{$\bullet$}};\node at  (27.8,0.5){$v_{14}$};
\end{tikzpicture}
\caption{An example of the involution $\vartheta$.\label{invo2}}
\end{figure}

Let $B\in\B_n$ be a binary tree.  We aim to construct an involution $\vartheta: \B_n\rightarrow\B_n$ that switches  the statistics ``$\spi$'' and ``$\rspi$'' as follows. Under the preorder, suppose that the $i$-th node of $B$ is $v_i$. Then,  the root of $\vartheta(B)$ is $v_n$ and for $i$ from $n-1$ down to $1$, 
\begin{itemize}
\item if  $v_i$ has right child $v_{\ell}$ for $\ell>i$ in $B$, then set $v_i$ to be the right child of $v_{\ell}$ in $\vartheta(B)$;
\item otherwise, $v_i$ has no right child  in $B$, then set $v_i$ to be the left child of  $v_{i+1}$ in $\vartheta(B)$. 
\end{itemize}
See Fig.~\ref{invo2} for an example of $\vartheta$. 
It is not immediately from the above simple construction that $\vartheta: \B_n\rightarrow\B_n$ is an involution. 
\begin{theorem}\label{thm:new}
The mapping $\vartheta: \B_n\rightarrow\B_n$ is an involution preserving  the right chain sequences and switching  the statistics ``$\spi$'' and ``$\rspi$''. 
\end{theorem}

\begin{proof}
Since the node $v_i$ of $B$ is the $(n+1-i)$-th node under the reverse preorder, the mapping $\vartheta: \B_n\rightarrow\B_n$ admits an alternative description. Under the reverse preorder, suppose that the $i$-th node of $B$ is $u_i$ (note that $u_i=v_{n+1-i}$). Then,  the root of $\vartheta(B)$ is $u_1$ and for $i$ from $2$ up to $n$, 
\begin{itemize}
\item if  $u_i$ has right child $u_{\ell}$ for $\ell<i$ in $B$, then set $u_i$ to be the right child of $u_{\ell}$ in $\vartheta(B)$;
\item otherwise, $u_i$ has no right child  in $B$, then set $u_i$ to be the left child of  $u_{i-1}$ in $\vartheta(B)$. 
\end{itemize}
It is clear from the above construction that $\rspi(B)=\spi(\vartheta(B))$. Thus, it remains to show that $\vartheta$ is an involution preserving  the right chain sequences. 

Note that a binary tree is completely determined by its left edges and right edges. Thus, it suffices to show that each edge of $B$ is  preserved  under $\vartheta^2$. By the construction of $\vartheta$, the node $v_i$ becomes the $i$-th node of $\vartheta(B)$ under the reverse preorder. If $v_j$ is the right child of $v_i$ (for some $i<j$) in $B$, then $v_i$ becomes the  right child of $v_j$ in $\vartheta(B)$ and so $v_j$ remains  the right child of $v_i$ in $\vartheta^2(B)$. In other words, the right edges are preserved under $\vartheta^2$. It also follows that $B$ and $\vartheta(B)$ have the same  right chain sequence. To see that  the left edges are preserved under $\vartheta^2$, suppose that $v_{i+1}$ is the left child of $v_i$ (for some $i$) in $B$. Note that  $v_{i+1}$ has no right child in $\vartheta(B)$, for otherwise   $v_{i+1}$ would be the right child of its parent in $B$, which contradicts with the fact that $v_{i+1}$ is the left child of $v_i$ in $B$. As $v_i$ and $v_{i+1}$ are respectively the $i$-th node and the $(i+1)$-th node in  $\vartheta(B)$ under the reverse preorder, it then follows from the alternative description  of $\vartheta$ that $v_{i+1}$ remains the left child of $v_i$ in $\vartheta^2(B)$.
Therefore, $\vartheta^2(B)=B$ and so $\vartheta$ is an involution on $\B_n$, as desired.  
\end{proof}

\begin{remark}
As the involution $\vartheta$ preserves the right edges of binary trees and di-sk trees are just binary trees labeled by 
$\oplus$ and $\ominus$ satisfying the right chain condition, $\vartheta$ is generalized naturally to di-sk trees which preserves the number of $\ominus$-nodes and switches the two statistics ``$\top$'' and ``$\rpop$''. 
This answers Problem~\ref{prob:bij}. 
\end{remark}

The  Catalan's triangle $C_{n,k}=\frac{k}{2n-k}{2n-k\choose n}$ has many known combinatorial interpretations in the literature (see~\cite{Cai,LL} and the references therein), among which is the enumeration of plane trees by the {\em length of  left arms}. 
The left chain staring from the root of $T$ is called the {\em left arm of $T$} and we denote by $\larm(T)$ the length of the left arm of $T$. 
Symmetrically, we define the {\em right arm of $T$} and $\rarm(T)$ the length  of the right arm of $T$.  In other words, $\rarm(T)$ is  the length of the left arm of the mirror symmetry of $T$. As an example, for the tree $T$ in Fig.~\ref{plan:preorder}, we have $\larm(T)=2$ and $\rarm(T)=3$. Denote by $\rev(T)$ the number of nodes before the first node (under the reverse preorder) that has younger sibling. By convention, if $T$ is the path with $n$ edges, then define $\rev(T)=n$. For example, if $T$ is the tree  in Fig.~\ref{plan:preorder}, then the node $13$ is the first node (under the reverse preorder) that has younger sibling and so $\rev(T)=5$ (counts the five nodes in blue).

The {\em degree of a node} in a plane tree is the number of its children. The multiset of the degrees of all internal  nodes of a plane tree $T$ is called the {\em degree sequence of $T$}.  The involution $\tilde\vartheta:\P_n\rightarrow\P_n$ defined by $\tilde\vartheta=\varphi^{-1}\circ\vartheta\circ\varphi$ has the following interesting property.  
\begin{proposition}\label{sym:rev}
The involution $\tilde\vartheta:\P_n\rightarrow\P_n$ preserves the degree sequences and the statistic ``$\rarm$'' but switches the statistics ``$\larm$'' and ``$\rev$''. In particular, ``$\rev$'' is a new tree statistic that interprets the Catalan triangle $C_{n,k}$. 
\end{proposition}

The proof of Proposition~\ref{sym:rev} is decomposed into the following two lemmas. For a binary tree $B$, the path starting from the root and ending at the last node (under preorder) is called the {\em right boundary of $B$}. Let $\lrb(B)$ be the number of {\bf l}eft edges in the {\bf r}ight {\bf b}oundary of $B$. Then, the following property of $\varphi$ can be checked routinely. 
\begin{lemma}\label{lem:var}
For any $T\in\P_n$, we have 
$$\larm(T)=\spi(\varphi(T)),\quad\rev(T)=\rspi(\varphi(T))\quad \text{and}\quad \rarm(T)=\lrb(\varphi(T))+1.
$$ 
Moreover, the degree sequence of $T$ equals the right chain sequence of $\varphi(T)$. 
\end{lemma}
Combining Theorem~\ref{thm:new} and Lemma~\ref{lem:var} with the following lemma  proves Proposition~\ref{sym:rev}.

\begin{lemma}
For any $B\in\B_n$, we have $\lrb(B)=\lrb(\vartheta(B))$. 
\end{lemma}
\begin{proof}
If $v$ is the left child of $u$ in the right boundary of  $B$, then $u$ has no right child in $B$. Thus, $u$ becomes  the left child of $v$ in the right boundary of  $\vartheta(B)$ (since $v$ has no right child). On the other hand, if $v$ is the right child of $u$ in the right boundary of  $B$, then obvisously  $u$ becomes the right child  of $v$ in the right boundary of  $\vartheta(B)$.  This proves $\lrb(B)=\lrb(\vartheta(B))$. 
\end{proof}

Proposition~\ref{sym:rev} implies that the refinement of Catalan numbers 
$$
C_n(x,y,z):=\sum_{T\in\P_n}x^{\larm(T)-1}y^{\rarm(T)-1} z^{\rev(T)-1}
$$
is symmetric in $x$ and $z$. In fact, $C_n(x,y,z)$ is a symmetric function in $x,y$ and $z$, as will be proved in next section. The first few values of $C_n(x,y,z)$ are listed as follows:
\begin{align*}
C_1(x,y,z)&=1,\\
C_2(x,y,z)&=1+xyz=1+m_{111},\\
C_3(x,y,z)&=1+xy+xz+yz+x^2y^2z^2=1+m_{11}+m_{222},\\
C_4(x,y,z)&=1+m_1+m_{11}+m_{22}+m_{211}+m_{333},\\
C_5(x,y,z)&=2+2m_{1}+m_{2}+2m_{11}+m_{21}+2m_{111}+m_{22}+m_{211}+m_{33}+m_{321}+m_{222}+m_{444},
\end{align*}
where $m_{\lambda}=m_{\lambda}(x,y,z)$ is the monomial symmetric function for a partition $\lambda$.

\begin{figure}
\centering
\begin{tikzpicture}[scale=0.32]
\draw[-](0,9) to (-3,6);\draw[-](0,9) to (0,6);\draw[-](0,9) to (9,0);
\draw[-](-3,6) to (-6,3);\draw[-](-3,6) to (0,3);\draw[-](0,3) to (0,0);
\draw[-](-3,6) to (-3,3);\draw[-](-3,0) to (-3,3);
\draw[-](-3,0) to (-5,-3);\draw[-](-3,0) to (-1,-3);\draw[-](1,0) to (1,-3);
\draw[-](3,6) to (3,3);\draw[-](3,3) to (1,0);\draw[-](3,3) to (5,0);\draw[-](5,-3) to (5,0);
\node at  (3,3){$\bullet$};\node at  (2,3.4){$12$};
\node at  (1,0){\red{$\bullet$}};\node at  (2,0){$13$};
\node at  (9,0){\blue{$\bullet$}};\node at  (10,0.4){$18$};
\node at  (1,-3){\blue{$\bullet$}};\node at  (2,-3){$14$};
\node at  (5,0){\blue{$\bullet$}};\node at  (6,0){$15$};
\node at  (5,-3){\blue{$\bullet$}};\node at  (6,-3){$16$};
\node at  (0,9){$\bullet$};\node at  (-0.8,9.2){$1$};
\node at  (-3,6){$\bullet$};\node at  (-3.6,6.4){$2$};
\node at  (0,6){$\bullet$};\node at  (-1,6){$10$};
\node at  (3,6){$\bullet$};\node at  (4,6.4){$11$};
\node at  (6,3){\blue{$\bullet$}};\node at  (7,3.4){$17$};
\node at  (-6,3){$\bullet$};\node at  (-6,2.2){$3$};
\node at  (0,3){$\bullet$};\node at  (-0.6,2.8){$8$};
\node at  (-3,3){$\bullet$};\node at  (-3.6,3){$4$};
\node at  (-3,0){$\bullet$};
\node at  (-5,-3){$\bullet$};\node at  (-5.8,-3){$6$};
\node at  (-1,-3){$\bullet$};\node at  (-1.8,-3){$7$};
\node at  (0,0){$\bullet$};\node at  (-0.6,-0.2){$9$};
\node at  (-3.5,0.5){$5$};

\node at  (14,3){$\longrightarrow$};
\node at  (14,3.8){$\tau$};


\draw[step=1,color=gray](18,-5) grid (35,12);
\draw[color=magenta](18,-5)--(35,12);
\draw[very thick](18,-5)--(21,-5)--(21,-4)--(24,-4)--(24,-2)--(25,-2)--(25,-1)--(27,-1)--(27,2)--(28,2)--(28,5)--(30,5)--(30,6)--(32,6)--(32,7)--(33,7)--(33,9)--(34,9)--(34,11)--(35,11)--(35,12);

\end{tikzpicture}
\caption{A plane tree labeled by its preorder and its associated Dyck path under $\tau$.\label{plan:preorder}}
\end{figure}
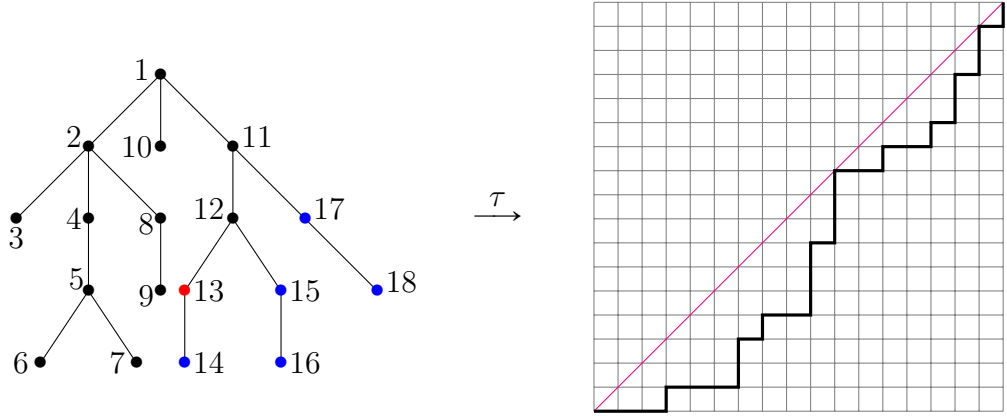

\subsection{A quadruple equidistribution on plane trees}
For a plane tree $T$, let $\run(T)$ be the number of non-leaf nodes after the last node (included) in preorder that has at least two children. By convention, if $T$ is the path with $n$ edges, then define $\run(T)=n$. For example, if $T$ is the tree in Fig.~\ref{plan:preorder}, then the $12$-th node is the last node (in preorder) that has at least two children and so $\run(T)=\#\{12,13,15,17\}=4$. This statistic originated from Duarte and Guedes de Oliveira's study~\cite{DG} of parking functions (see also~\cite{LL}). The main result in this section is the following  bijection on plane trees. 

\begin{theorem}\label{thm:psi}
There exists a bijection $\Psi: \P_n\rightarrow \P_n$ that preserves  the triple of statistics  
$(\larm,\rarm,\leaf)$ and transforms the statistic ``$\rev$'' to ``$\run$''. Consequently, $C_n(x,y,z)$ is a symmetric function in $x,y$ and $z$. 
\end{theorem}
\begin{remark}
The distribution of the pair $(\rev,\run)$ is not symmetric over $\P_5$.
\end{remark}

Before the construction of the bijection $\Psi$, we show how the symmetry of $C_n(x,y,z)$ follows from $\Psi$ and an involution on Dyck paths introduced in~\cite{LL}. 

Recall that a {\em Dyck path} of order $n$ is a lattice path in $\N^2$ from $(0,0)$ to
$(n,n)$ using the {\em east step} $(1,0)$ and the {\em north step} $(0,1)$, which does
not pass above the diagonal $y=x$. Let $\mathcal{D}_n$ be the set of all Dyck paths of order $n$. Let ${\bf c}=(c_1,c_2,\ldots,c_k)$ be a composition of $n$. A Dyck path  $D\in\D_n$   is said to has {\em composition-type ${\bf c}$} if 
it begins with $c_1$ east steps followed by at least one north step, and then continues with $c_2$ east steps followed by at least one north step, and then so on. For instance, the Dyck path in Fig.~\ref{plan:preorder} has composition type $(3,3,1,2,1,2,2,1,1,1)$. Three statistics introduced in~\cite{LL} are involved:
\begin{itemize}
\item the number of times that $D$ returns to the diagonal $y=x$ after the departure, denoted $\ret(D)$; 
\item the number of north steps before the first north step (included) that followed immediately by a north step, denoted $\hrun(D)$;  
\item the number of east steps after the last east step (excluded) that followed immediately by an east step, denoted $\vrun(D)$. 
\end{itemize}
By convention, for the only zigzag Dyck path (i.e., each east step is followed by a north step) $D\in\D_n$, we define $\hrun(D)=\vrun(D)=n$. For example,  for the Dyck path $D$ in Fig.~\ref{plan:preorder}, we have $\ret(D)=3$, $\hrun(D)=2$ and $\vrun(D)=4$. 
The following involution on Dyck paths was constructed by Li and Lin~\cite{LL}.
\begin{theorem}[Li and Lin]
There exists a composition-type preserving involution $\Phi$ on $\D_n$ that exchanges the pair of statistics $(\hrun,\ret)$. 
\end{theorem}

\begin{corollary}\label{sym:dyck}
The three-variable function 
$$
\sum_{D\in\D_n}x^{\hrun(D)} y^{\ret(D)}z^{\vrun(D)}
$$ is a symmetric function. 
\end{corollary}
\begin{proof}
Since the composition-type of a Dyck path determines the statistic `$\vrun$', the symmetry 
$$
\sum_{D\in\D_n}x^{\hrun(D)} y^{\ret(D)}z^{\vrun(D)}=\sum_{D\in\D_n}y^{\hrun(D)} x^{\ret(D)}z^{\vrun(D)}
$$
then follows from the involution $\Phi$. The other two symmetries are clear by combining with  the involution of taking the mirror symmetry of the Dyck paths after rotating 180 degrees. 
\end{proof}

There is a bijection $\tau$ (see~\cite{LL}) that maps a plane tree to a Dyck path by recording the steps when the tree is traversed in preorder: whenever we visit  a node of degree $i$ (for the first time), except the last leaf, we record $i$ east steps followed by one north step.  See Fig.~\ref{plan:preorder} for one example of $\tau$. The bijection $\tau$ has the following properties,  which are clear from the above construction. 

\begin{lemma}\label{lem:tau}For any plane tree $T$, we have
\begin{equation}\label{fea:tau}
(\larm,\rarm,\run)(T)=(\hrun,\ret,\vrun)\tau(T).
\end{equation}
\end{lemma}

The fact that $C_n(x,y,z)$ is a symmetric function then follows by combining  Lemma~\ref{lem:tau}, Corollary~\ref{sym:dyck} and the bijection $\Psi$. The rest of this section is devoted to the construction of $\Psi$, which is based on a  technical lemma regarding the statistic ``$\rev$''. 

For a plane tree $T$, let $\bran(T)$ be  the number of branches of $T$ at the root, i.e., the number of children of the root. Let 
$$\P_n^{(k,l)}:=\{T\in\P_n: \bran(T)\geq2, \larm(T)=k,\rarm(T)=l\}.$$

\begin{lemma}\label{lem:rev}
Fix integers $n,k\geq1$. There exists a bijection $\psi:\P_n^{(k,2)}\rightarrow\P_n^{(k+1,1)}$ preserving the pair of statistics $(\rev,\leaf)$. 
\end{lemma}
\begin{proof}
For the sake of convenience, we call the first node (under the reverse preorder) that has younger sibling in $T$ the {\em rev-node of $T$}. Thus, $\rev(T)$ counts the nodes before the rev-node of $T$ under the reverse preorder.  The construction of $\psi$ consists of two main steps: (i) shift  the branches of the nodes in the right and left arms of $T$; (ii) adjust the resulting tree if necessary to keep the statistic ``$\rev$''. 

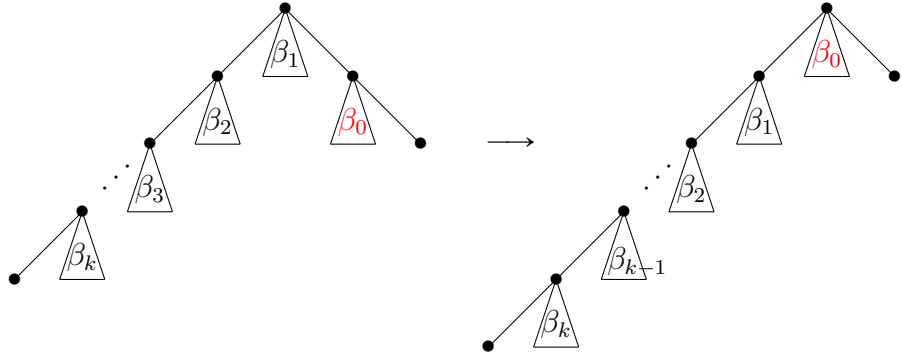
\begin{figure}
\begin{center}
\begin{tikzpicture}[scale=0.3]
\draw[-](0,0) to (-6,-6);\draw[-](0,0) to (6,-6);\draw[-](-12,-12) to (-9,-9);
\draw[-](-3,-3) to (-2,-6);\draw[-](-3,-3) to (-4,-6);\draw[-](-4,-6) to (-2,-6);
\draw[-](0,0) to (1,-3);\draw[-](0,0) to (-1,-3);\draw[-](1,-3) to (-1,-3);

\draw[-](3,-3) to (2,-6);\draw[-](3,-3) to (4,-6);\draw[-](2,-6) to (4,-6);
\draw[-](-6,-6) to (-7,-9);\draw[-](-6,-6) to (-5,-9);\draw[-](-5,-9) to (-7,-9);
\draw[-](-9,-9) to (-8,-12);\draw[-](-9,-9) to (-10,-12);\draw[-](-8,-12) to (-10,-12);\node at  (-9,-11){$\beta_{k}$};
\node at  (0,0){$\bullet$};
\node at  (-3,-3){$\bullet$};\node at  (-3,-5){$\beta_2$};
\node at  (-6,-6){$\bullet$};\node at  (-6,-8){$\beta_3$};
\node at  (3,-3){$\bullet$};\node at  (3,-5){\red{$\beta_0$}};\node at  (0,-2){$\beta_1$};
\node at  (6,-6){$\bullet$};

\node at  (-9,-9){$\bullet$};
\node at  (-12,-12){$\bullet$};
\node at  (-7,-7){$\cdot$};\node at  (-7.5,-7.5){$\cdot$};\node at  (-8,-8){$\cdot$};
\node at  (10,-6){$\longrightarrow$};


\draw[-](24,0) to (18,-6);\draw[-](24,0) to (27,-3);\draw[-](12,-12) to (15,-9);
\draw[-](21,-3) to (22,-6);\draw[-](21,-3) to (20,-6);\draw[-](20,-6) to (22,-6);
\draw[-](24,0) to (25,-3);\draw[-](24,0) to (23,-3);\draw[-](25,-3) to (23,-3);

\draw[-](18,-6) to (17,-9);\draw[-](18,-6) to (19,-9);\draw[-](19,-9) to (17,-9);
\draw[-](15,-9) to (16,-12);\draw[-](15,-9) to (14,-12);\draw[-](16,-12) to (14,-12);\node at  (15.6,-11.2){$\beta_{k-1}$};
\node at  (24,0){$\bullet$};
\node at  (21,-3){$\bullet$};\node at  (21,-5){$\beta_1$};
\node at  (18,-6){$\bullet$};\node at  (18,-8){$\beta_2$};
\node at  (27,-3){$\bullet$};\node at  (24,-2){\red{$\beta_0$}};

\node at  (15,-9){$\bullet$};\node at  (12,-14){$\beta_k$};
\node at  (12,-12){$\bullet$};\draw[-](12,-12) to (11,-15);\draw[-](12,-12) to (13,-15);\draw[-](13,-15) to (11,-15);
\draw[-](12,-12) to (9,-15);\node at  (9,-15){$\bullet$};
\node at  (17,-7){$\cdot$};\node at  (16.5,-7.5){$\cdot$};\node at  (16,-8){$\cdot$};

\end{tikzpicture}
\end{center}
\caption{The first step of $\psi$.\label{step1}}
\end{figure}

{\bf Step (i) of $\psi$.} Given $T\in\P_n^{(k,2)}$, let $\beta_0$ be the subtree at the second node in the right arm and let $\beta_i$ ($1\leq i\leq k$) be the subtree at the $i$-th node in the left arm,  after deleting all the edges in the left and right arms of $T$. Then form the plane tree $T'\in\P_n^{(k+1,1)}$ such that after deleting all the edges in the left and right arms, $\beta_i$ ($0\leq i\leq k$) is the subtree at the $(i+1)$-th node in the left arm of $T'$. See Fig.~\ref{step1} for an illustration of this shifting $T\mapsto T'$. 

If $\beta_0$ contains at least two nodes, then $\rev(T)=\rev(T')$ and we set $\psi(T)=T'$. Otherwise, $\beta_0$ contains only one node and we perform  the step (ii) of $\psi$.

{\bf Step (ii) of $\psi$.} As $\beta_0$ contains only one node, $\rev(T')$ is not always equal to $\rev(T)$ and so we need to distinguish two main cases. 
\begin{itemize}
\item[a)] $\rev(T')\geq \rev(T)$. We further distinguish two subcases:
\begin{itemize}
\item[a1)]{\bf$\beta_1$ contains only one node.} This forces all $\beta_i$, $2\leq i\leq k$, contain only one node and thus $\rev(T')=\rev(T)$. We set $\psi(T)=T'$ in this case.  
\item[a2)] {\bf $\beta_1$  contains at least two nodes.}  Thus, the subtree at the youngest child of the root of $\beta_1$ must be a path (otherwise, we will have $\rev(T')<\rev(T)$, a contradiction) that is denoted by $P$. Suppose that $T^*$ is the subtree of $T'$ at the second node in the left arm after cutting the branch $P$.  Then $\psi(T)$ is obtained from $T'$ by cutting  the path $P$  and attaching it to the node which is the penultimate node on the right arm of $T^*$ (see Fig.~\ref{step2:a} for an example). It is clear that $\rev(T)=\rev(\psi(T))$ holds.
 Note that  $\psi(T)=T'$ may occur when the penultimate node on the right arm of $T^*$ is the parent of the root of $P$ in $T'$, i.e., when $\rarm(T^*)=1$. 
 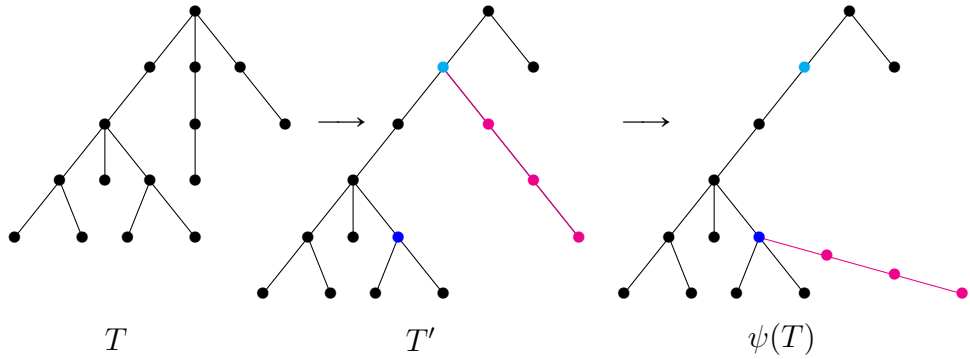
\begin{figure}
\begin{center}
\begin{tikzpicture}[scale=0.3]
\draw[-](0,0) to (-8,-10);
\node at  (0,0){$\bullet$};
\node at  (-2,-2.5){$\bullet$};
\node at  (-4,-5){$\bullet$};\node at  (-6,-7.5){$\bullet$};\node at  (-8,-10){$\bullet$};
\draw[-](0,0) to (4,-5);\node at  (4,-5){$\bullet$};\node at  (2,-2.5){$\bullet$};
\draw[-](0,0) to (0,-7.5);
\node at  (0,-2.5){$\bullet$};\node at  (0,-5){$\bullet$};\node at  (0,-7.5){$\bullet$};
\draw[-](-4,-5) to (0,-10);\node at  (-2,-7.5){$\bullet$};\node at  (0,-10){$\bullet$};
\draw[-](-2,-7.5) to (-3,-10);\node at  (-3,-10){$\bullet$};
\draw[-](-4,-5) to (-4,-7.5);\node at  (-4,-7.5){$\bullet$};
\draw[-](-6,-7.5) to (-5,-10);\node at  (-5,-10){$\bullet$};
\node at  (6.5,-5){$\longrightarrow$};
\node at  (-3.5,-14.5){$T$};
\draw[-](13,0) to (3,-12.5);\draw[-](13,0) to (15,-2.5);\draw[-](5,-10) to (6,-12.5);\draw[-](7,-7.5) to (7,-10);\draw[-](7,-7.5) to (11,-12.5);\draw[-](9,-10) to (8,-12.5);\draw[-](11,-2.5) to (17,-10);
\draw[magenta](11,-2.5) to (17,-10);
\node at  (13,0){$\bullet$};
\cyan{\node at  (11,-2.5){$\bullet$};}
\node at  (15,-2.5){$\bullet$};
\node at  (9,-5){$\bullet$};\node at  (7,-7.5){$\bullet$};\node at  (5,-10){$\bullet$};\node at  (3,-12.5){$\bullet$};\node at  (6,-12.5){$\bullet$};
\node at  (7,-10){$\bullet$};\node at  (9,-10){\blue{$\bullet$}};\node at  (11,-12.5){$\bullet$};\node at  (8,-12.5){$\bullet$};\magenta{\node at  (13,-5){$\bullet$};\node at  (15,-7.5){$\bullet$};\node at  (17,-10){$\bullet$};}
\node at  (10,-14.5){$T'$};
\node at  (20,-5){$\longrightarrow$};

\draw[-](29,0) to (19,-12.5);\draw[-](29,0) to (31,-2.5);\draw[-](21,-10) to (22,-12.5);\draw[-](23,-7.5) to (23,-10);\draw[-](23,-7.5) to (27,-12.5);\draw[-](25,-10) to (24,-12.5);
\node at  (29,0){$\bullet$};
\cyan{\node at  (27,-2.5){$\bullet$};}
\node at  (31,-2.5){$\bullet$};
\node at  (25,-5){$\bullet$};\node at  (23,-7.5){$\bullet$};
\node at  (21,-10){$\bullet$};\node at  (19,-12.5){$\bullet$};\node at  (22,-12.5){$\bullet$};
\node at  (23,-10){$\bullet$};
\node at  (27,-12.5){$\bullet$};\node at  (24,-12.5){$\bullet$};
\draw[magenta](25,-10) to (34,-12.5);
\node at  (28,-10.83){\magenta{$\bullet$}};\node at  (31,-11.66){\magenta{$\bullet$}};\node at  (34,-12.5){\magenta{$\bullet$}};
\node at  (25,-10){\blue{$\bullet$}};
\node at  (26,-14.5){$\psi(T)$};
\end{tikzpicture}
\end{center}
\caption{An example of case a2).\label{step2:a}}
\end{figure}
\end{itemize}
\item[b)] $\rev(T')<\rev(T)$. In this case, the parent of the rev-node of $T$ can not be the root. Thus, we have $\rev(T')=\rev(T)-1$. We further distinguish three subcases:
 \begin{figure}
\begin{center}
\begin{tikzpicture}[scale=0.3]
\draw[-](0,0) to (-6,-7.5);\draw[-](0,0) to (4,-5);\draw[-](-2,-2.5) to (2,-7.5);
\node at  (0,0){$\bullet$};\node at  (-2,-2.5){$\bullet$};
\node at  (-4,-5){$\bullet$};\node at  (-6,-7.5){$\bullet$};
\node at  (2,-2.5){$\bullet$};\node at  (4,-5){$\bullet$};
\node at  (0,-5){$\bullet$};\node at  (2,-7.5){$\bullet$};
\draw[-](-4,-5) to (-4,-7.5);\draw[-](-4,-5) to (0,-10);\draw[-](-2,-7.5) to (-10,-17.5);
\node at  (-4,-7.5){$\bullet$};\node at  (-2,-7.5){$\bullet$};\node at  (0,-10){$\bullet$};
\node at  (-4,-10){$\bullet$};\node at  (-6,-12.5){$\bullet$};\node at  (-8,-15){$\bullet$};
\node at  (-10,-17.5){$\bullet$};\draw[-](-4,-10) to (-7,-11);\node at  (-7,-11){$\bullet$};
\draw[-](-6,-12.5) to (-4,-15);\node at  (-4,-15){$\bullet$};
\node at  (-4,-20){$T$};
\node at  (5,-8){$\longrightarrow$};
\draw[-](9,-7.5) to (9,-10);\draw[cyan](11,-10) to (13,-12.5);
\draw[-](9,-7.5) to (11,-10);\draw[-](11,-10) to (3,-20);\draw[-](9,-12.5) to (6,-13.5);
\draw[-](15,0) to (7,-10);\draw[-](15,0) to (17,-2.5);\draw[magenta](11,-5) to (15,-10);
\draw[blue](7,-15) to (9,-17.5);\draw[very thick,gray](13,-2.5) to (9,-7.5);
\draw[very thick,gray](11,-10) to (9,-7.5);\draw[very thick,gray](11,-10) to (7,-15);

\node at  (15,0){$\bullet$};
\red{\node at  (13,-2.5){$\bullet$};}\node at  (12.3,-2.3){$v$};
\node at  (11,-5){$\bullet$};
\node at  (9,-7.5){$\bullet$};\node at  (7,-10){$\bullet$};\node at  (17,-2.5){$\bullet$};
\magenta{\node at  (13,-7.5){$\bullet$};\node at  (15,-10){$\bullet$};}
\node at  (9,-10){$\bullet$};
\node at  (11,-10){$\bullet$};\node at  (13,-12.5){\cyan{$\bullet$}};
\node at  (9,-12.5){$\bullet$};
\node at  (6,-13.5){$\bullet$};
\red{\node at  (7,-15){$\bullet$};}\blue{\node at  (9,-17.5){$\bullet$};}
\node at  (5,-17.5){$\bullet$};
\node at  (3,-20){$\bullet$};
\node at  (10,-20){$T'$};
\node at  (19,-8){$\longrightarrow$};

\draw[-](30,0) to (22,-10);\draw[-](30,0) to (32,-2.5);\draw[magenta](28,-2.5) to (32,-7.5);
\draw[-](24,-7.5) to (24,-10);\node at  (24,-10){$\bullet$};
\draw[cyan](24,-7.5) to (28,-10);
\draw[-](24,-7.5) to (26,-10);
\draw[-](24,-12.5) to (21,-13.5);
\draw[-](26,-10) to (18,-20);

\draw[very thick,gray](28,-2.5) to (24,-7.5);
\draw[very thick,gray](26,-10) to (24,-7.5);\draw[very thick,gray](26,-10) to (22,-15);

\node at  (30,0){$\bullet$};
\red{\node at  (28,-2.5){$\bullet$};}\node at  (27.3,-2.3){$v$};
\node at  (26,-5){$\bullet$};
\node at  (24,-7.5){$\bullet$};\node at  (22,-10){$\bullet$};\node at  (32,-2.5){$\bullet$};
\magenta{\node at  (30,-5){$\bullet$};\node at  (32,-7.5){$\bullet$};}

\node at  (26,-10){$\bullet$};
\node at  (28,-10){\cyan{$\bullet$}};
\node at  (24,-12.5){$\bullet$};
\node at  (21,-13.5){$\bullet$};\draw[blue](24,-12.5) to (26,-15);
\red{\node at  (22,-15){$\bullet$};}\blue{\node at  (26,-15){$\bullet$};}
\node at  (20,-17.5){$\bullet$};
\node at  (18,-20){$\bullet$};
\node at  (26,-20){$\psi(T)$};

\end{tikzpicture}
\end{center}
\caption{An example of case b1).\label{step2:b1}}
\end{figure}
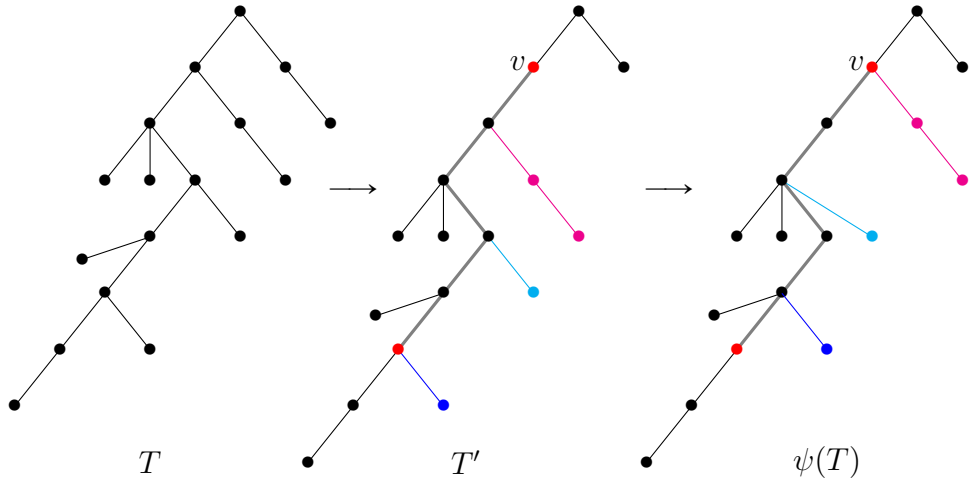

\begin{itemize}
\item[b1)] {\bf The rev-node of $T'$ has no elder sibling.} Let $P$ be the unique path in $T'$ from the second node $v$ in the left arm  to the parent of the rev-node (see the path in gray of $T'$ in Fig.~\ref{step2:b1}). Then, each node in $P$ has at most one branch, which must be a path, to the right of $P$ (see the three paths in color of $T'$ in Fig.~\ref{step2:b1}). In particular, there is no any branch of $v$ to the right of $P$. Now define $\psi(T)$ to be the tree obtained from $T'$ by moving all the paths connected to $P$ (from the right side) one step up. See  Fig.~\ref{step2:b1} for an example in this case. It is clear that $\rev(\psi(T))=\rev(T')+1=\rev(T)$. 

 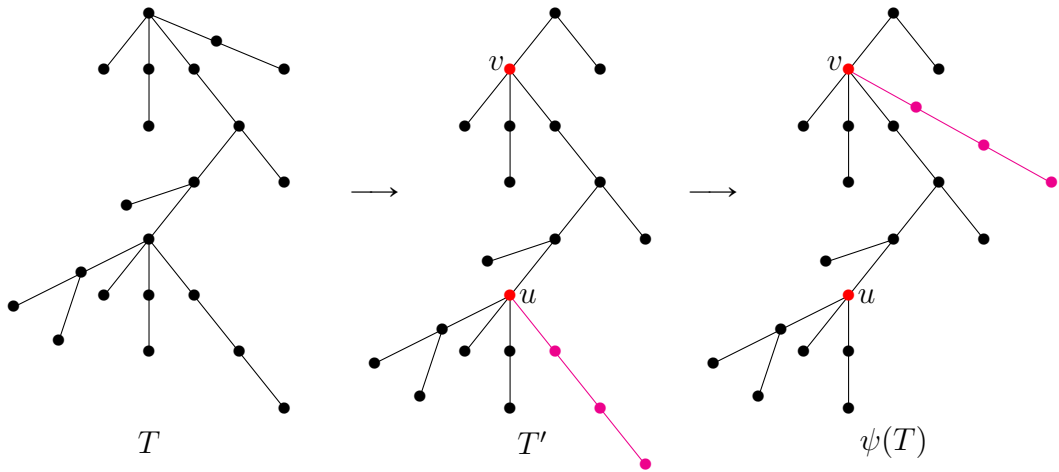
\begin{figure}
\begin{center}
\begin{tikzpicture}[scale=0.3]
\draw[-](0,0) to (-2,-2.5);\draw[-](0,0) to (0,-5);\draw[-](0,0) to (6,-7.5);\draw[-](0,0) to (6,-2.5);
\node at  (0,0){$\bullet$};\node at  (-2,-2.5){$\bullet$};\node at  (0,-2.5){$\bullet$};\node at  (0,-5){$\bullet$};\node at  (2,-2.5){$\bullet$};\node at  (4,-5){$\bullet$};\node at  (6,-7.5){$\bullet$};
\node at  (3,-1.25){$\bullet$};\node at  (6,-2.5){$\bullet$};
\draw[-](4,-5) to (-2,-12.5);\draw[-](0,-10) to (6,-17.5);\draw[-](0,-10) to (0,-15);
\draw[-](0,-10) to (-6,-13);\draw[-](-3,-11.5) to (-4,-14.5);\draw[-](2,-7.5) to (-1,-8.5);
\node at  (0,-10){$\bullet$};
\node at  (2,-7.5){$\bullet$};\node at  (-1,-8.5){$\bullet$};
\node at  (2,-12.5){$\bullet$};\node at  (4,-15){$\bullet$};\node at  (6,-17.5){$\bullet$};
\node at  (0,-12.5){$\bullet$};\node at  (0,-15){$\bullet$};
\node at  (-2,-12.5){$\bullet$};\node at  (-3,-11.5){$\bullet$};\node at  (-6,-13){$\bullet$};
\node at  (-4,-14.5){$\bullet$};

\node at  (0,-19){$T$};
\node at  (10,-8){$\longrightarrow$};

\draw[-](16,-2.5) to (14,-5);\draw[-](16,-2.5) to (16,-7.5);\draw[-](16,-2.5) to (22,-10);
\draw[-](16,-2.5) to (18,0);\draw[-](20,-2.5) to (18,0);
\draw[-](20,-7.5) to (14,-15);\draw[magenta](16,-12.5) to (22,-20);\draw[-](16,-12.5) to (16,-17.5);
\draw[-](16,-12.5) to (10,-15.5);\draw[-](13,-14) to (12,-17);\draw[-](18,-10) to (15,-11);
\node at  (18,0){$\bullet$};\node at  (20,-2.5){$\bullet$};
\red{\node at  (16,-2.5){$\bullet$};}\node at  (15.4,-2.2){$v$};
\node at  (14,-5){$\bullet$};
\node at  (16,-5){$\bullet$};\node at  (16,-7.5){$\bullet$};\node at  (18,-5){$\bullet$};\node at  (20,-7.5){$\bullet$};\node at  (22,-10){$\bullet$};

\node at  (16,-12.5){\red{$\bullet$}};\node at  (16.8,-12.5){$u$};
\node at  (18,-10){$\bullet$};\node at  (15,-11){$\bullet$};
\magenta{\node at  (18,-15){$\bullet$};\node at  (20,-17.5){$\bullet$};\node at  (22,-20){$\bullet$};}
\node at  (16,-15){$\bullet$};\node at  (16,-17.5){$\bullet$};
\node at  (14,-15){$\bullet$};\node at  (13,-14){$\bullet$};\node at  (10,-15.5){$\bullet$};
\node at  (12,-17){$\bullet$};
\node at  (17,-19){$T'$};
\node at  (25,-8){$\longrightarrow$};

\draw[-](31,-2.5) to (29,-5);\draw[-](31,-2.5) to (31,-7.5);\draw[-](31,-2.5) to (37,-10);
\draw[magenta](31,-2.5) to (40,-7.5);
\draw[-](31,-2.5) to (33,0);\draw[-](35,-2.5) to (33,0);
\draw[-](35,-7.5) to (29,-15);\draw[-](31,-12.5) to (31,-17.5);
\draw[-](31,-12.5) to (25,-15.5);\draw[-](28,-14) to (27,-17);\draw[-](33,-10) to (30,-11);
\node at  (33,0){$\bullet$};\node at  (35,-2.5){$\bullet$};
\red{\node at  (31,-2.5){$\bullet$};}\node at  (30.4,-2.2){$v$};
\node at  (29,-5){$\bullet$};\node at  (31,-5){$\bullet$};\node at  (31,-7.5){$\bullet$};\node at  (33,-5){$\bullet$};\node at  (35,-7.5){$\bullet$};\node at  (37,-10){$\bullet$};
\magenta{\node at  (34,-4.166){$\bullet$};\node at  (37,-5.832){$\bullet$};\node at  (40,-7.5){$\bullet$};}

\node at  (31,-12.5){\red{$\bullet$}};\node at  (31.8,-12.5){$u$};
\node at  (33,-10){$\bullet$};\node at  (30,-11){$\bullet$};
\node at  (31,-15){$\bullet$};\node at  (31,-17.5){$\bullet$};
\node at  (29,-15){$\bullet$};\node at  (28,-14){$\bullet$};\node at  (25,-15.5){$\bullet$};
\node at  (27,-17){$\bullet$};
\node at  (33,-19){$\psi(T)$};
\end{tikzpicture}
\end{center}
\caption{An example of case b2).\label{step2:b2}}
\end{figure}

\item[b2)]{\bf The rev-node of $T'$ has  elder siblings and the closest elder sibling has no children.}
As in case b1), let $P$ be the unique path in $T'$ from the second node $v$ in the left arm  to the parent $u$ of the rev-node. Then, each node in $P$ has at most one branch, which must be a path, to the right of $P$. In particular, there is no any branch of $v$ to the right of $P$. Now define $\psi(T)$ to be the tree obtained from $T'$ by cutting the rightmost branch of $u$ and attaching it to $v$ as rightmost branch. See  Fig.~\ref{step2:b2} for an example in this case. It is routine to check that $\rev(\psi(T))=\rev(T')+1=\rev(T)$. 

\item[b3)] {\bf The rev-node of $T'$ has  elder siblings and the closest elder sibling has  children.}
Let $u$ be the parent of the rev-node in $T'$. Let $P$ be the branch  of the subtree at $u$ that contains the rev-node of $T'$. Then, $P$ must be a path. Let $B$ be the third branch (from right to left) of the subtree at $u$. Now define $\psi(T)$ to be the tree obtained from $T'$ by cutting the path $P$ and attaching it to the penultimate node on the right arm of $B$. See  Fig.~\ref{step2:b3} for an example in this case. The reader is invited  to check that $\rev(\psi(T))=\rev(T')+1=\rev(T)$.

 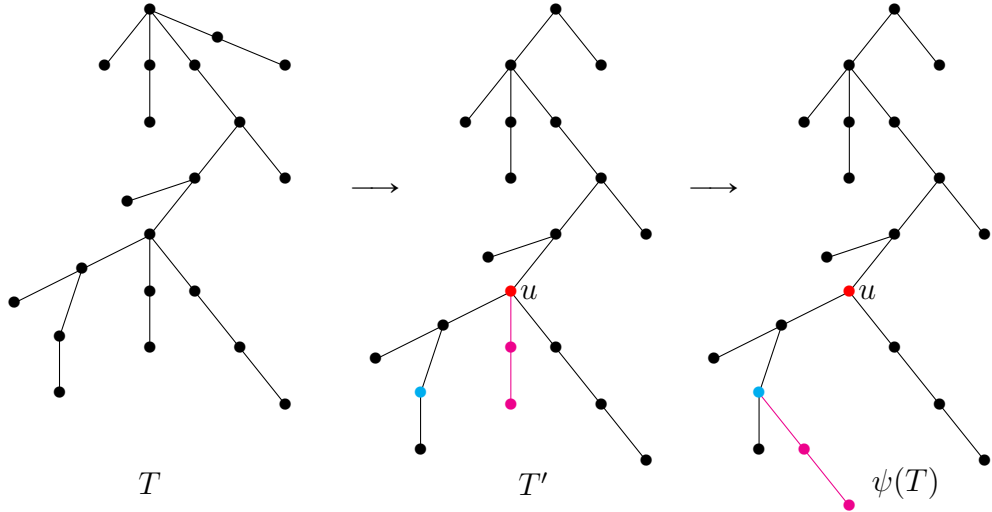
\begin{figure}
\begin{center}
\begin{tikzpicture}[scale=0.3]
\draw[-](0,0) to (-2,-2.5);\draw[-](0,0) to (0,-5);\draw[-](0,0) to (6,-7.5);\draw[-](0,0) to (6,-2.5);
\node at  (0,0){$\bullet$};\node at  (-2,-2.5){$\bullet$};\node at  (0,-2.5){$\bullet$};\node at  (0,-5){$\bullet$};\node at  (2,-2.5){$\bullet$};\node at  (4,-5){$\bullet$};\node at  (6,-7.5){$\bullet$};
\node at  (3,-1.25){$\bullet$};\node at  (6,-2.5){$\bullet$};
\draw[-](4,-5) to (0,-10);
\draw[-](0,-10) to (6,-17.5);\draw[-](0,-10) to (0,-15);
\draw[-](0,-10) to (-6,-13);\draw[-](-3,-11.5) to (-4,-14.5);\draw[-](2,-7.5) to (-1,-8.5);
\node at  (0,-10){$\bullet$};
\node at  (2,-7.5){$\bullet$};\node at  (-1,-8.5){$\bullet$};
\node at  (2,-12.5){$\bullet$};\node at  (4,-15){$\bullet$};\node at  (6,-17.5){$\bullet$};
\node at  (0,-12.5){$\bullet$};\node at  (0,-15){$\bullet$};

\draw[-](-4,-14.5) to (-4,-17);
\node at  (-3,-11.5){$\bullet$};
\node at  (-6,-13){$\bullet$};
\node at  (-4,-14.5){$\bullet$};
\node at  (-4,-17){$\bullet$};
\node at  (0,-21){$T$};
\node at  (10,-8){$\longrightarrow$};

\draw[-](16,-2.5) to (14,-5);\draw[-](16,-2.5) to (16,-7.5);\draw[-](16,-2.5) to (22,-10);
\draw[-](16,-2.5) to (18,0);\draw[-](20,-2.5) to (18,0);
\draw[-](20,-7.5) to (16,-12.5);\draw[-](16,-12.5) to (22,-20);\draw[magenta](16,-12.5) to (16,-17.5);
\draw[-](16,-12.5) to (10,-15.5);\draw[-](13,-14) to (12,-17);\draw[-](18,-10) to (15,-11);
\node at  (18,0){$\bullet$};\node at  (20,-2.5){$\bullet$};
\node at  (16,-2.5){$\bullet$};\node at  (14,-5){$\bullet$};\node at  (16,-5){$\bullet$};\node at  (16,-7.5){$\bullet$};\node at  (18,-5){$\bullet$};\node at  (20,-7.5){$\bullet$};\node at  (22,-10){$\bullet$};

\node at  (16,-12.5){\red{$\bullet$}};\node at  (16.8,-12.5){$u$};
\node at  (18,-10){$\bullet$};\node at  (15,-11){$\bullet$};
\node at  (18,-15){$\bullet$};\node at  (20,-17.5){$\bullet$};\node at  (22,-20){$\bullet$};
\magenta{\node at  (16,-15){$\bullet$};\node at  (16,-17.5){$\bullet$};}
\draw[-](12,-17) to (12,-19.5);
\node at  (13,-14){$\bullet$};
\node at  (10,-15.5){$\bullet$};
\node at  (12,-17){\cyan{$\bullet$}};
\node at  (12,-19.5){$\bullet$};
\node at  (17,-21){$T'$};
\node at  (25,-8){$\longrightarrow$};

\draw[-](31,-2.5) to (29,-5);\draw[-](31,-2.5) to (31,-7.5);\draw[-](31,-2.5) to (37,-10);

\draw[-](31,-2.5) to (33,0);\draw[-](35,-2.5) to (33,0);
\draw[-](35,-7.5) to (31,-12.5);
\draw[-](31,-12.5) to (25,-15.5);\draw[-](28,-14) to (27,-17);\draw[-](33,-10) to (30,-11);
\node at  (33,0){$\bullet$};\node at  (35,-2.5){$\bullet$};
\node at  (31,-2.5){$\bullet$};\node at  (29,-5){$\bullet$};\node at  (31,-5){$\bullet$};\node at  (31,-7.5){$\bullet$};\node at  (33,-5){$\bullet$};\node at  (35,-7.5){$\bullet$};\node at  (37,-10){$\bullet$};

\draw[-](31,-12.5) to (37,-20);\draw[magenta](27,-17) to (31,-22);
\magenta{\node at  (29,-19.5){$\bullet$};\node at  (31,-22){$\bullet$};}
\node at  (33,-15){$\bullet$};\node at  (35,-17.5){$\bullet$};\node at  (37,-20){$\bullet$};
\node at  (33,-10){$\bullet$};\node at  (30,-11){$\bullet$};
\draw[-](27,-17) to (27,-19.5);
\node at  (28,-14){$\bullet$};\node at  (25,-15.5){$\bullet$};
\node at  (27,-17){\cyan{$\bullet$}};
\node at  (27,-19.5){$\bullet$};
\node at  (31,-12.5){\red{$\bullet$}};\node at  (31.8,-12.5){$u$};
\node at  (33.5,-21){$\psi(T)$};
\end{tikzpicture}
\end{center}
\caption{An example of case b3).\label{step2:b3}}
\end{figure}
\end{itemize}
\end{itemize}

It is clear from the above construction that the map $\psi:\P_n^{(k,2)}\rightarrow\P_n^{(k+1,1)}$ preserves the number of leaves. It remains to show that $\psi$ is a bijection. We observe that each of the five cases in the construction of $\psi$ is reversible
and so we only need to figure out which case is involved for each specified tree. Given $T\in\P_n^{(k+1,1)}$, let us decompose $T$ as in the right tree in Fig.~\ref{step1}. If $\beta_0$ contains at least two nodes, then we apply the reverse operation of step (i) to retrieve $\psi^{-1}(T)$ directly. Otherwise, $\beta_0$ contains only one  node and we can  retrieve $T'$ and then  $\psi^{-1}(T)$ as follows. 

Suppose that $v$ is the second node in the left arm of $T$. If $v$ is the rev-node, then $T'=T$. This  corresponds to step (ii) case a1). Otherwise, $v$ is not the rev-node and let $P$ be the path from $v$ to the parent $u$ of the rev-node of $T$. For the sake of convenience, each extra branch (to the right of $P$) from a certain node $x$ on $P$ is called the ``hanging branch'' of $x$, and denoted as $H_x$ (note that $H_x$ must be a path). When such a hanging branch does not exist, we simply write $H_x=\emptyset$. So by the definition of $P$, the node $u$ (the parent of the rev-node) must have $H_u\neq\emptyset$.   We can retrieve $T'$ in the following way:
\begin{itemize}
\item If the rev-node has no children and each node in $P$ other than $u$ has no hanging branch, then apply the reverse  operation of step (ii) case a2) to get $T'$. In fact, $T'$ is obtained from $T$ by cutting  $H_u$ from $u$ and attaching it to $v$ so that it becomes the hanging branch for $v$. 
\item If the rev-node has no children and $H_v\neq\emptyset$, then apply the reverse  operation of step (ii) case b2) to get $T'$. In this case, let $T'$ be the tree obtained from $T$ by cutting $H_v$ from $v$ and attaching it to $u$ so that it becomes the new hanging branch for $u$.
\item If the rev-node has no children and it is not in the above two situations, then apply the reverse  operation of step (ii) case b3) to get $T'$.  Suppose that $x$ is the closest node to $u$ on $P$ with the property that $H_x\neq\emptyset$. We cut $H_u$ and insert it to $x$ as a middle branch between $P$ and $H_x$. This new tree is then taken to be $T'$.
\item If the rev-node has one child, then apply the reverse  operation of step (ii) case b1) to get $T'$. Suppose that $r$ is the rev-node of $T$ and $s$ is the child of $r$. We cut $H_u$ from $u$ and attach it to $r$ so that it becomes the hanging branch for $r$ and the node $s$ becomes the new rev-node. For the remaining nodes on $P$, we ``slide down'' their hanging branches by one edge. Namely, if node $y$ succeeds $x$ on $P$, then we cut $H_x$ and attach it to $y$ so that it becomes the new hanging branch for $y$, and $H_y$ is slid down to $y$'s successor if any, and so on and so forth. The new tree obtained this way is set to be $T'$.
\end{itemize}
Since the above cases are disjoint and exhaust all possible $T\in\P_n^{(k+1,1)}$ whose $\beta_0$ contains only one  node, the map $\psi$ is a bijection. 
\end{proof}

We also need the following bijection for the statistic ``$\run$'' analogous to Lemma~\ref{lem:rev} for the statistic ``$\rev$''.

\begin{lemma}\label{lem:run}
Fix integers $n,k\geq1$. There exists a bijection $\eta:\P_n^{(k,2)}\rightarrow\P_n^{(k+1,1)}$ preserving the degree sequences and the triple of statistics $(\run,\leaf,\bran)$. 
\end{lemma}
\begin{proof}
For a Dyck path $D\in\D_n$, denote by $\iest(D)$  the number of initial consecutive east  steps  of $D$ and by $\peak(D)$ the number of peaks of $D$, where a peak of $D$ is an east step followed immediately by a north step. For each $T\in\P_n$, the bijection $\tau$ satisfies $\bran(T)=\iest(\tau(T))$ and $\leaf(T)=n+1-\peak(\tau(T))$. If we introduce
$$
\D_n^{(k,l)}:=\{D\in\D_n: \iest(D)\geq2,\hrun(D)=k, \ret(D)=l\},
$$
 then Li and Lin~\cite{LL} have constructed a  bijection $f:\D_n^{(k,2)}\rightarrow\D_n^{(k+1,1)}$ that preserves the composition-type and the triple  of statistics $(\vrun,\iest,\peak)$. It then follows from Lemma~\ref{lem:tau} that the bijection $\eta=\tau^{-1}\circ f\circ\tau$ has the required properties. 
\end{proof}

\begin{remark}
Small examples (e.g., $n=6,k=1$) show that Lemma~\ref{lem:rev} for the statistic ``$\rev$'' can not be refined further either by the number of branches or by the degree sequences.
\end{remark}

Now we are ready to construct recursively $\Psi$ to prove Theorem~\ref{thm:psi}. 

\begin{proof}[{\bf Proof of Theorem~\ref{thm:psi}}] We will construct $\Psi$ recursively with the help of Lemmas~\ref{lem:rev} and~\ref{lem:run}. Initially, the only tree in $\P_1$ is fixed under $\Psi$. For $n\geq2$, suppose that the map $\Psi$ is defined for $\P_{n-1}$. We perform the following three main cases to define $\Psi$ for each tree $T\in\P_n$.
\begin{enumerate}
\item $\bran(T)=1$. Let $T^*$ be the only branch of $T$. Define $\Psi(T)$ to be the tree with one brach $\Psi(T^*)$. Clearly, the triple of statistics $(\leaf,\larm,\rarm)$ is preserved and $\rev(T)=\rev(T^*)=\run(\Psi(T^*))=\run(\Psi(T))$ in this case.
\item $\bran(T)\geq2$ and $\rarm(T)\geq2$. Let $B$ be the subtree of $T$ rooted at  the penultimate node on the right arm of $T$. We further distinguish two subcases according to $B$:
\begin{itemize}
\item $\bran(B)\geq2$. Define $\Psi(T)$ to be the tree obtained  by replacing $B$ in $T$ by $\Psi(B)$.
Clearly, the triple of statistics $(\leaf,\larm,\rarm)$ is preserved and $\rev(T)=\rev(B)=\run(\Psi(B))=\run(\Psi(T))$ in this case. 
\item $\bran(B)=1$. Let $T^*$ be the tree obtained from $T$ by deleting the last node in the right arm. Define $\Psi(T)$ to be the tree $\Psi(T^*)$ after adding  to the last node in its right arm a child. Again, the triple of statistics $(\leaf,\larm,\rarm)$ is preserved and $\rev(T)=\rev(T^*)+1=\run(\Psi(T^*))+1=\run(\Psi(T))$ in this case. 
\end{itemize}
\item $\bran(T)\geq2$ and $\rarm(T)=1$. Suppose that $\larm(T)=k+1$ for some $k\geq0$. We further distinguish two subcases according to the value $k$:
\begin{itemize}
\item $k=0$. In this case, let $T^*$ be the tree obtained from $T$ by deleting the eldest child of its root. Then define $\Psi(T)$ to be the tree obtained from $\Psi(T^*)$ by attaching a child to its root as eldest child. Clearly, the triple of statistics $(\leaf,\larm,\rarm)$ is preserved and 
$\rev(T)=\rev(T^*)=\run(\Psi(T^*))=\run(\Psi(T))$ in this case.
\item $k\geq1$. In this case, $T$ is a tree in $\P_n^{(k+1,1)}$ and so $\psi^{-1}(T)\in\P_n^{(k,2)}$, which is a tree in case~(2) above. Let $T'=\Psi(\psi^{-1}(T))\in\P_n^{(k,2)}$ be obtained in case~(2) above. Then define $\Psi(T)=\eta(T')\in\P_n^{(k+1,1)}$. By Lemmas~\ref{lem:rev} and~\ref{lem:run}, we see that the triple of statistics $(\leaf,\larm,\rarm)$ is preserved and $\rev(T)=\rev(\psi^{-1}(T))=\run(T')=\run(\Psi(T))$ in this case. 
\end{itemize} 
\end{enumerate}

It is easy to check that $\Psi$ is a bijection by induction, which completes the proof of Theorem~\ref{thm:psi}.
\end{proof}

\subsection{Generalized to rooted  labeled trees}
In this section, we show that the involution $\vartheta$ on binary trees and $\tilde\vartheta$ on plane trees can be generalized to right increasing binary trees and rooted labeled trees, respectively. 

A {\em right increasing binary tree} on $[n]$ is a binary tree labeled by $[n]$ such that the label of each right child is greater than its parent. Let $\mathcal{RB}_n$ be the set of all right increasing binary trees on $[n]$. 
A {\em rooted labeled tree} on $[n]_0:=\{0,1,2,\ldots,n\}$ is a plane tree labeled by $[n]_0$ such that the root is labeled $0$ and the labels of children of any node is increasing from left to right. 
Let $\mathcal{LT}_n$ be the set of all rooted labeled trees on $[n]_0$. 
The bijection $\varphi: \P_n\rightarrow\B_n$ defined in Section~\ref{sec:2} can be extended naturally to $\varphi: \mathcal{LT}_n\rightarrow\mathcal{RB}_n$. See Fig.~\ref{rl:tree} for an example of $\varphi$.
 Since it is well known~\cite{Yan} that the cardinality of $\mathcal{LT}_n$ is $(n+1)^{n-1}$, then so is $\mathcal{RB}_n$.

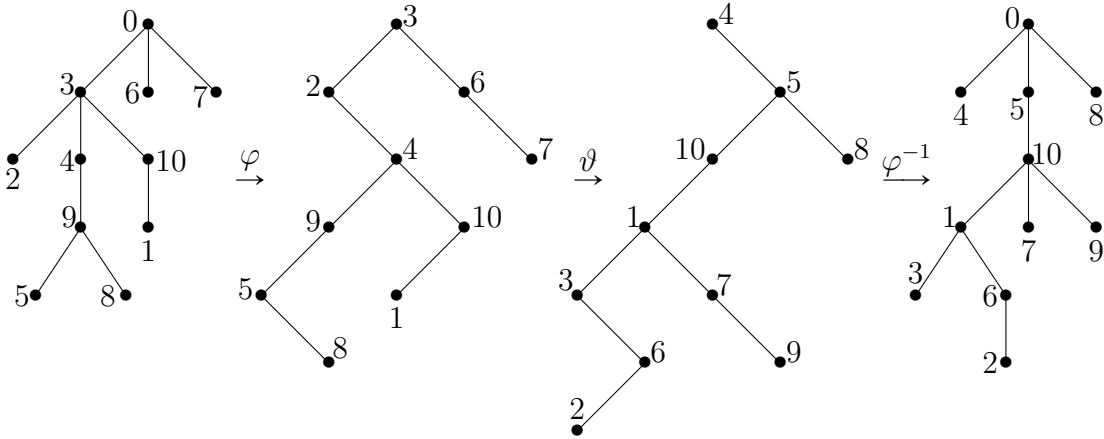
\begin{figure}
\begin{center}
\begin{tikzpicture}[scale=0.3]
\draw[-](0,9) to (-3,6);\draw[-](0,9) to (0,6);\draw[-](0,9) to (3,6);
\draw[-](-3,6) to (-6,3);\draw[-](-3,6) to (0,3);\draw[-](0,3) to (0,0);
\draw[-](-3,6) to (-3,3);\draw[-](-3,0) to (-3,3);
\draw[-](-3,0) to (-5,-3);\draw[-](-3,0) to (-1,-3);
\node at  (0,9){$\bullet$};\node at  (-0.8,9.2){$0$};
\node at  (-3,6){$\bullet$};\node at  (-3.6,6.4){$3$};
\node at  (0,6){$\bullet$};\node at  (-0.7,6){$6$};
\node at  (3,6){$\bullet$};\node at  (2.3,5.8){$7$};
\node at  (-6,3){$\bullet$};\node at  (-6,2.2){$2$};
\node at  (0,3){$\bullet$};\node at  (1,3){$10$};
\node at  (-3,3){$\bullet$};\node at  (-3.6,3){$4$};
\node at  (-3,0){$\bullet$};
\node at  (-5,-3){$\bullet$};\node at  (-5.6,-3){$5$};
\node at  (-1,-3){$\bullet$};\node at  (-1.8,-3){$8$};
\node at  (0,0){$\bullet$};\node at  (0,-1){$1$};
\node at  (-3.5,0.5){$9$};

\node at  (4.5,2){$\rightarrow$};\node at  (4.5,3){$\varphi$};

\draw[-](11,9) to (8,6);
\draw[-](11,9) to (17,3);\draw[-](8,6) to (14,0);
\draw[-](11,-3) to (14,0);\draw[-](11,3) to (5,-3);\draw[-](8,-6) to (5,-3);
\node at  (11,9){$\bullet$};\node at  (11.6,9.4){$3$};
\node at  (8,6){$\bullet$};\node at  (7.3,6.3){$2$};
\node at  (14,6){$\bullet$};\node at  (14.6,6.5){$6$};
\node at  (17,3){$\bullet$};\node at  (17.6,3.5){$7$};
\node at  (11,3){$\bullet$};\node at  (11.6,3.5){$4$};
\node at  (14,0){$\bullet$};\node at  (15,0.5){$10$};
\node at  (11,-3){$\bullet$};\node at  (11,-4){$1$};
\node at  (8,0){$\bullet$};\node at  (7.3,0.3){$9$};
\node at  (5,-3){$\bullet$};\node at  (4.3,-2.7){$5$};
\node at  (8,-6){$\bullet$};\node at  (8.5,-5.3){$8$};

\node at  (19.5,2){$\rightarrow$};\node at  (19.5,3){$\vartheta$};

\draw[-](25,9) to (31,3);\draw[-](28,6) to (19,-3);\draw[-](22,0) to (28,-6);
\draw[-](19,-3) to (22,-6);\draw[-](22,-6) to (19,-9);

\node at  (19,-9){$\bullet$};\node at  (19,-8){$2$};
\node at  (22,-6){$\bullet$};\node at  (22.6,-5.5){$6$};
\node at  (25,9){$\bullet$};\node at  (25.6,9.5){$4$};
\node at  (28,6){$\bullet$};\node at  (28.6,6.5){$5$};
\node at  (31,3){$\bullet$};\node at  (31.6,3.5){$8$};
\node at  (25,3){$\bullet$};\node at  (24,3.5){$10$};
\node at  (22,0){$\bullet$};\node at  (21.5,0.5){$1$};
\node at  (19,-3){$\bullet$};\node at  (18.5,-2.3){$3$};
\node at  (25,-3){$\bullet$};\node at  (25.5,-2.5){$7$};
\node at  (28,-6){$\bullet$};\node at  (28.6,-5.5){$9$};

\node at  (33.6,2){$\longrightarrow$};\node at  (33.6,3){$\varphi^{-1}$};

\draw[-](39,9) to (36,6);\draw[-](39,9) to (42,6);\draw[-](39,9) to (39,6);
\draw[-](39,0) to (39,6);\draw[-](39,3) to (36,0);\draw[-](39,3) to (42,0);

\draw[-](36,0) to (34,-3);\draw[-](36,0) to (38,-3);\draw[-](38,-6) to (38,-3);

\node at  (34,-3){$\bullet$};\node at  (34,-2){$3$};
\node at  (38,-3){$\bullet$};\node at  (37.3,-3){$6$};
\node at  (38,-6){$\bullet$};\node at  (37.3,-6){$2$};
\node at  (39,9){$\bullet$};\node at  (38.3,9.3){$0$};
\node at  (36,6){$\bullet$};\node at  (36,5){$4$};
\node at  (42,6){$\bullet$};\node at  (42,5){$8$};
\node at  (39,6){$\bullet$};\node at  (38.5,5.3){$5$};
\node at  (39,0){$\bullet$};\node at  (39,-1){$7$};
\node at  (39,3){$\bullet$};\node at  (39.8,3.3){$10$};
\node at  (36,0){$\bullet$};\node at  (35.5,0.5){$1$};
\node at  (42,0){$\bullet$};\node at  (42,-1){$9$};

\end{tikzpicture}
\end{center}
\caption{An example of  the involution $\tilde\vartheta=\varphi\circ\vartheta\circ\varphi^{-1}$.\label{rl:tree}}
\end{figure}

 We observe that our involution $\vartheta: \B_n\rightarrow\B_n$ can be generalized to $\mathcal{RB}_n$ as follows. Given $T\in\mathcal{RB}_n$ whose underlying  binary tree is $B$, define $\vartheta(T)$ to be the right increasing binary tree whose underlying binary tree is $\vartheta(B)$ and a node is labeled by $i$ in $\vartheta(T)$ iff it is labeled by $n+1-i$ in $T$. See Fig.~\ref{rl:tree} for an example of $\vartheta$.
 
 Since the labeling is irrelevant for the seven tree statistics $(\spi,\rspi,\lrb,\larm,\rarm,\rev,\run)$, they can be extended naturally either to right increasing binary trees or to rooted labeled trees.
 Exactly the same discussions as in the proof of Theorem~\ref{thm:new} results in the following analogue. 
 \begin{theorem}\label{thm:newt}
The involution $\vartheta: \mathcal{RB}_n\rightarrow\mathcal{RB}_n$ preserves  the right chain sequences and switches   the statistics ``$\spi$'' and ``$\rspi$''. 
\end{theorem}

Introduce the refinement $T_{n,k}$ of the tree function $(n+1)^{n-1}$ by the exponential generating function 
\begin{equation*}
\sum_{n\geq0}\sum_{k=1}^nT_{n,k}t^k\frac{x^n}{n!}=\biggl(1-t\sum_{n\geq1}(n-1)^{n-1}\frac{x^n}{n!}\biggr)^{-1}.
\end{equation*}
Among several other combinatorial interpretations found in~\cite{DG,Du,LL}, it was known that $T_{n,k}$ enumerates rooted labeled trees $T\in\mathcal{LT}_n$ with $\larm(T)=k$ (or $\rarm(T)=k$). As one application of Theorem~\ref{thm:newt}, we have more interpretations for the triangle $T_{n,k}$ summarized  
in the following result, whose proof is the same as that for plane trees and will be omitted.  
\begin{proposition}\label{sym:revt}
 We have the following results concerning the bijection $\varphi$ and the involution $\tilde\vartheta:=\varphi\circ\vartheta\circ\varphi^{-1}$ (see Fig.~\ref{rl:tree} for an example of $\tilde\vartheta$).
\begin{itemize}
\item[(i)] For any $T\in\mathcal{LT}_n$, we have 
$$\larm(T)=\spi(\varphi(T)),\quad\rev(T)=\rspi(\varphi(T))\quad \text{and}\quad \rarm(T)=\lrb(\varphi(T))+1.
$$ 
\item[(ii)]
The involution $\tilde\vartheta:\mathcal{LT}_n\rightarrow\mathcal{LT}_n$ preserves the degree sequences and the statistic ``$\rarm$'' but switches the statistics ``$\larm$'' and ``$\rev$''. 
\item[(iii)] Each statistic in $\{\spi,\rspi,\rev,\lrb\}$ is a new tree statistic that interprets the  triangle $T_{n,k}$. 
\end{itemize}
\end{proposition}

Regarding the counterpart of Theorem~\ref{thm:psi} for $\mathcal{LT}_n$, we have the following observations. 
\begin{itemize}
\item The statistic ``$\rev$'' is not equidistributed with ``$\run$'' over $\mathcal{LT}_4$.
\item The triple $(\larm,\rarm,\rev)$ is not equidistributed with $(\rarm,\larm,\rev)$ over $\mathcal{LT}_4$.
\item Proposition~\ref{sym:revt}~(ii) implies that the pair $(\rarm,\rev)$ has the same distribution as $(\rarm,\larm)$ over $\mathcal{LT}_n$. As the pair $(\rarm,\larm)$ is symmetric by the mirror symmetry of rooted labeled trees, we see that $(\rarm,\rev)$ is symmetric over $\mathcal{LT}_n$. 
\end{itemize}

\section*{Acknowledgement}

The authors thank Joanna N. Chen and Jing Liu for their  discussions that leads to the finding of Theorem~\ref{thm:lsw}. The authors also wish to thank the anonymous referee for carefully reading the paper and providing insightful comments and suggestions.
This work was supported by the National Science Foundation of China grants 12322115, 12271301 and 12201641,  and the project of Qilu Young Scholars of Shandong University.

\end{document}